\documentclass[12pt,a4paper,twoside]{amsart}
\usepackage{amsmath}
\usepackage{amssymb}
\usepackage{amsfonts}
\usepackage{a4}

\usepackage{tikz}
\tikzset{node distance=3cm, auto}

\numberwithin{equation}{section}

\theoremstyle{plain}

\newtheorem{theorem}[subsection]{Theorem}
\newtheorem{definition}[subsection]{Definition}
\newtheorem{proposition}[subsection]{Proposition}
\newtheorem{lemma}[subsection]{Lemma}
\newtheorem{remark}[subsection]{Remark}

\newtheorem{corollary}[subsection]{Corollary}

\newtheorem{claim}[subsection]{Claim}

\theoremstyle{definition}

\newcommand{\2}{{\bf 2}}

\newcommand{\mA}{\mathbb A}

\newcommand{\mC}{{\mathbb C}}

\newcommand{\mE}{{\mathbb E}}
\newcommand{\mF}{\mathbb F}

\newcommand{\mP}{\mathbb P}
\newcommand{\mQ}{\mathbb Q}

\newcommand{\mZ}{{\mathbb Z}}

\newcommand{\ho}{\hookrightarrow}
\newcommand{\Gg}{\gamma}

\newcommand{\bo}{\omega}

\newcommand{\kk}{\kappa}

\newcommand{\mcA}{\mathcal A}

\newcommand{\ti}{\tilde}

\newcommand{\emp}{\emptyset}

\usepackage{color}

\begin{document}

\title{Extending linear and quadratic functions from high rank varieties.}
\begin{abstract}
Let $k$ be a field, $V$ be a $k$-vector space and $X\subset V$ an  algebraic irreducible subvariety.

We say that a  function $f:X(k) \to k$ is {\it weakly linear}
 if its restriction  to any  two-dimensional linear subspace $W$ of $V$ contained in $X$ is linear and that it is {\it weakly quadratic} if its restriction  to any  three-dimensional linear subspace $W$ of $V$ contained in $X$ is quadratic.
  We 
say that $X$ is {\it admissible} if
 any weakly linear function on $X$ is a restriction of a linear function on $V$
 and  any weakly quadratic function on $X$ is a restriction of a quadratic function on $V$.

The main result in the paper concerns the case when the field $k$ is a finite.  We show that for any $d,L\geq 1$  there exists  $r=r(d,L,k)\in \mZ _+$ such that any homogeneous complete intersection  $X\in V$ in a vector space $V$ of codimension $L$, degree $d$ and  rank $\geq r$ is admissible.

Moreover we show the existence of a function   $r(d,L)$ such that one can take $r(d,L,k)=r(d,L)$ for all finite fields  $k$  of characteristic $>d$.

The proof of the {\it admissibility} for finite fields $k$ is based on bounds on the number of $k$-points on ancillary varieties $E(X)$. These results allow us to bound the dimension of varieties $E(X)$.
 Using these results we were able to prove  the {\it admissibility} 
of complex homogeneous varieties of high rank.

Using the results of \cite{br} one can extend our proofs to show   the  {\it admissibility} of varieties of high rank over local non-archimedian fields.   Also using  Corollary $4.3$ of \cite{cmpv} one can dispense with the assumption that $X$ is a complete intersection. \\

\end{abstract}

\author{David Kazhdan}
\address{Einstein Institute of Mathematics,
Edmond J. Safra Campus, Givaat Ram 
The Hebrew University of Jerusalem,
Jerusalem, 91904, Israel}
\email{david.kazhdan@mail.huji.ac.il}

\author{Tamar Ziegler}
\address{Einstein Institute of Mathematics,
Edmond J. Safra Campus, Givaat Ram 
The Hebrew University of Jerusalem,
Jerusalem, 91904, Israel}
\email{tamarz@math.huji.ac.il}

\thanks{The second author is supported by ERC grant ErgComNum 682150. Part of the material in this paper is based upon work  supported by the National Science Foundation under Grant No. DMS-1440140 while the second author was in residence at the Mathematical Sciences Research Institute in Berkeley, California, during the Spring 2017 semester.}
\maketitle
\tableofcontents

\section{Introduction}

Let $k$ be a field. Let $V$ be a $k$-vector space.

\begin{definition}[Rank]\label{alg} a) Let $P\in k[V^\vee ]$  be a non-zero  polynomial  of degree $d\geq 2$. We 
define the {\em rank} $r(P)$ of $P$ as  the minimal number $r$ such that it is possible to write $P$ in the form
$P = \sum ^ r_{i=1} l_iR_i,$ where $l_i,R_i$ are  polynomials of positive degrees. \\
b) Given a family $\bar P=\{P_i\},i\in I$ of  polynomials of degree $\le d$ 
we define the rank of $\bar P$ as the minimal rank of a non-zero linear combination of polynomials $P_i,i\in I$.
\end{definition}

\begin{definition}[Degree]\label{degree} 
Let $X\subset V$ be  an irreducible  $k$-subvariety $X\subset V$ of codimension $L$ and $\bar X \subset \mP (V)$ 
be the corresponding projective variety. We define the {\it degree} $d(X)$  of $X$ as the number of $\bar k$-points in the intersection $\bar X\cap \mP (M)(\bar k)$ for generic subspaces $M$ of $V$ of dimension $L$ defined over over $\bar k$.
\end{definition}

\begin{definition}[Admissible]\label{admissible} 
A subset $X\subset V$ is {\it admissible} if  any weakly linear function $f:X \to k$ is linear.  \end{definition}

Fix $d,L\geq 1$. Our main theorems are the following: 

\begin{theorem}\label{finite-multi}
a) Let $k$ be a finite field. There exists $r=r(d,L,k)$ such that any  homogeneous complete intersection $X \subset V$ of degree $d$, codimension $L$ and of  rank $>r$ is {\it admissible}. \\
b) There exists a function   $r(d,L)$ such that one can take $r(d,L,k)=r(d,L)$ for all finite fields  $k$  of characteristic $>d$.
\end{theorem}

\begin{theorem}\label{complex-multi}
Let $k=\mathbb C$. There exists $r=r(d,L)$ such that any  homogeneous  complete intersection $X \subset V$ of degree $d$, codimension $L$ and of  rank $>r$ is {\it admissible}
\end{theorem}

\begin{remark} In all  places when a statement is true in the case when the rank  $r$ is larger then some function $r(d,L,k)$ we may assume that the function $r(d,L,k)$ does not depend on $k$ if characteristic of $k$ is $>d$.
\end{remark} 

\begin{definition}[Quadratically-admissible]\label{2admissible} 
A subset $X\subset V$ is {\it quadratically-admissible} if  any weakly quadratic function $f:X \to k$ is quadratic.  \end{definition}

\begin{theorem}\label{quadratic-extension}
For any $d\geq 1$ and   finite field $k$ there exists $r>0$ such that any $k$-vector space $V$, any 
 homogeneous hypersurface $X \subset V$ of degree $d$ and rank $>r$ is quadratically admissible.
\end{theorem}

\begin{remark} The proof extends without difficulty to the case where $X$ is a homogeneous  complete intersection of a fixed codimension $L$:
Fix $L\geq 1$. Let $k$ be a finite field. There exists $r=r(L,d)$ such that any homogeneous  complete intersection $X \subset V$ of codimension $L$ and of  rank $>r$ is {\it admissible}.  In the case when $d=2$ the $r$ depends linearly on $L$. 
\end{remark}

\begin{theorem}\label{quadratic-extension-complex}
Let $k=\mathbb C$. There exists $r$ such that any homogeneous hypersurface $X\subset V$ of  rank $>r$ is  quadratically admissible. \\
\end{theorem}

\begin{remark}
 \cite{cmpv} implies the  existence of  a bound on the number and on degrees of polynomials generating a homogeneous subvariety $X$ of $V$ in terms of the degree and the codimension of $X$. We can therefore dispense with the assumption that $X$ is a complete intersection. We are thankful to Mel Hochster for a reference to \cite{cmpv}.
\end{remark}

%%%%%%%%%%%%%%%%%%%%%%
\section{Notation and key ingredients}

For a finite set $X$ we denote $\mE_{x \in X}$ the average $|X|^{-1}\sum_{x \in X}$. \\

We say that a map $p:X\to Y$ between finite sets is {\it $C$-homogeneous}, $C\geq 1$ 
if 
$|S_y|/|S_{y'}|\leq C, y,y'\in Y$ where $S_y:=p^{-1}(y)$.\\

If $Y^0\subset Y$ are  finite sets we say that {\em a property $P:\{ y\in Y^0\}$ is satisfied $\epsilon$-a.e } if 
$|Y^0|\geq  (1-\epsilon)|Y|$.  If $Y^0\subset Y$ are  algebraic $\mathbb C$-variety we say that {\em a property  $P:\{ y\in Y^0\}$ is satisfied $\kk$-a.e } if $Y-Y^0\subset Y$ is a  subvariety of codimension $\ge \kk$.  \\

 Below are some  lemmas are finitary analogues of measure theoretic properties. 
 
 We say that a map $p:X\to Y$ between finite sets is {\it $C$-homogeneous}, $C\geq 1$ 
if 
$|S_y|/|S_{y'}|\leq C, y,y'\in Y$ where $S_y:=p^{-1}(y)$.
Let $p:Z\to Y$ be a {\it $C$-homogeneous} map and $\epsilon $ be a positive number $\leq C^{-2}$.

Let $P\subset Z$ be a subset 
such that $|P|/|Z|\geq 1-\epsilon$. We define $Q\subset Y$ by 
$$Q=\{y\in Y: |S_y\cap P|/|S_y|\geq 1-C^2\epsilon /2\}$$

\begin{lemma}[Fubini]\label{fubini}  $|Q|/|Y|\geq 1-C^2\epsilon/2$.

\end{lemma}

\begin{proof}
Let $A=P^c$, and let $A_{y}=  S_{y} \cap P^c$.  Let $B=Q^c$, and let $S=|S_{y_0}|$ for some $y_0 \in Y$. Then
\[
\epsilon CS|Y|\ge  \epsilon|Z| \ge |B|S/C + (|Y|-|B|) C^2\epsilon S/2C
\]
so that
\[
\epsilon C |Y|/2 \ge |B|(1/C-\epsilon C/2) 
\]
If $\epsilon C<1/C$ we get  $|B|< |Y| C^2\epsilon/2 $.
\end{proof}

The next Lemmas are  immediate: 

\begin{lemma}\label{density-ae} Let $B \subset A$ with $|B| \ge c|A|$ suppose a property $P$ holds $\epsilon$ a.e. $x \in A$ then 
$P$ holds for $\epsilon/c$-a.e. $x \in B$.  
\end{lemma}
 
\begin{lemma}\label{ae-projection}
Let $p: X \to Z$ by a map of finite sets. Let $Y \subset Z$ be such that  $|Y|/|Z|\ge 1-\epsilon$, the restriction of $p$ to 
$p^{-1}(Y)$ is $c$-homogeneous , namely  there is an $s>0$ so that $c^{-1}s\le |S_y-s| \le  cs$, for all $y\in Y$, and assume further that for all $z \in Z$
$|S_z|\le Cs$ for some $C>0$.  Then for any $Z' \subset Z$ with $|Z'|<\delta |Z|$ we have $|X'|=p^{-1}|Z'| < \delta \frac{C}{c(1-\epsilon)}|X|$.
For any subset $A \subset X$ we have $|A\cap X'|/|A| \le \delta \frac{C}{c(1-\epsilon)}|X|/|A|$.
\end{lemma}

Lemma \ref{ae-projection} allows us to pull back good properties of an image of a map to the source.

\begin{definition}
For any  $m\in \mZ_{>0}$ we  define $\2^m :=\{0,1\}^m$. We define $| \bo | =\sum _1^m\bo _i ,\bo \in \2^m$ and say that $\bo$ is even (odd) if   $| \bo | $ is even (odd).  Let $k$ be a field and $V$ be a $k$-vector space.  For any $\bo \in \2^m, \bar v =\{ v_1, \ldots, v_m \}\in V^m$ we define 
$\omega \cdot \bar v :=\sum _1^m \bo _iv_i\in V$.

For any  $u\in V,\bar v =\{ v_1, \ldots, v_m \}\in V^m$ we denote by $\phi _{(u|\bar v )}:2^m\to G$ the map  given by 
$$\phi _{(u|\bar v )}(\bo):=u+\omega \cdot \bar v $$
and denote by  $(u|\bar v)\subset V$ the image of the $\phi _{(u|\bar v )}$
 and  $(u|\bar v)'\subset V$ the image of the restriction of  $\phi _{(u|\bar v )}$ to $2^m \setminus \{ 0\}$. We say that  the subsets of $V$ of the form $(u|\bar v)$ are {\it $m$-cubes} and that  subsets $V$ of the form $(u|\bar v)'$ are {\it almost cubes}.

For a cube $(u|\bar v)$ we call $u$ the {\em base} and $v_1, \ldots, v_m$ the {\em generators}.

For a subset  $X \subset V$ we denote  $C_m(X)$ the set of $m$-cubes in $V$ with all vertices in $X$ and  $C'_m(X)$ the set of {\it almost cubes} in $V$ with all vertices in $X$.
 
Let $H$ be an abelian group.  For any  $H$-valued  function $f$ on $X$ we denote by $f_m$ the function on $C_m(X)$ defined by
\[
f_m(u|\bar v) =\sum_{\omega \in \2^m}-1^{|\omega|} f(u+\omega\cdot \bar v)
\]
and  by $f'_m$ the function on $C'_m(X)$ defined by
\[
f'_m(u|\bar v) =\sum_{\omega \in \2^m \setminus \{ 0\}}-1^{|\omega|} f(u+\omega\cdot \bar v).
\]
We say that $c \in C_m(X)$ is {\em good for $f$} if $f_m(c)=0$. 

Given a function $f:X\to H$ we write $f_m(X)=0$ if all $c\in C_m(X)$ are good for $f$.
\end{definition}

Example: 
For $v,v_1, v_2\in V$, $(v|v_1,v_2)$ is the  {\it parallelogram}
$$(v|v_1, v_2) = (v,v+v_1,v+v_2, v+v_1+v_2).$$

From now untill  Appendix $2$ we assume that $k=\mF _q$, $q=p^l$ and denote by $e_q:k\to \mC^\star$ the additive character 
$$e_q (a):=\exp (tr _{k/\mF _p}(a))$$ 
where $\exp :\mF _q\to \mC ^\star$ is a non-trivial additive character.\\

For $f:V \to k$ we define the $U_2$ norm of $f$  by 
\[
\|f\|^4_{U^2} = | \mE_{v, v_1, v_2 \in V}e_q(f_2(v|v_1,v_2))|.
\]

\begin{lemma}[Gowers-CS \cite{gowers}]\label{gowers-cs}Let $f_i:V \to k$, $i=1, \ldots, 4$
\[
|\mE_{v, v_1, v_2 \in V} e_q(f_1(v)+f_2(v+v_1)+f_3(v+v_2)+ f_4(v+v_1+v_2))| \le \min \|f_i\|_{U_2}.
\]
\end{lemma}

\begin{definition}[Bias]
The {\em bias} of a function $f:V\to k$ is defined by 
\[
b(f)=\text{bias}(f)=|\mE_{x}e_q(f(x))|.
\]
\end{definition}

\begin{definition}[Measurability-Rank] Let $P:V \to k$ be a  polynomial, and $d \ge 2$ an integer. The {\em $d$-measurable-rank} of $P$, denoted meas-$r_d(P)$, is the smallest integer $r$ such that there exist polynomials $Q_1, \ldots, Q_r$ of degree $\le d-1$ and a function $\Gamma: V \to k$ with $P=\Gamma(Q_1, \ldots, Q_r)$. If $d=1$ then meas-$r_1(P)$ is defined to be $\infty$ if $P$ is non constant and $0$ otherwise. The rank of a polynomial $P$ is denoted meas-$r(P)$ and defined to be meas-$r_{\text{deg}(P)}(P)$. 
\end{definition} 

\begin{remark}[\cite{BL}]\label{rank-alg-rank}  Many results in the literature are stated in terms of the measurable rank (and it is called rank there). We prefer to use the algebraic notion of rank  introduced in \ref{alg}. It is shown in [\cite{BL}] that these definitions are essentially the equivalent  for prime fields $k$: For any $d\ge 2$, there exists a function $R_d(r)$ such that  the following hold: for any prime field $k=\mF_p$, and $k$-vector space $W$, any $P\in k[W^\vee ]$ of degree $d$ with meas-$r(P)=r$ has rank  $\le R_d(r)$.  The same holds true  for general fields as long as the characteristic is greater than the degree. The other direction is obvious.
\end{remark}

A key ingredient of our proof is the following relation between the notions of {\it bias} and {\it rank}.

\begin{theorem}[Bias implies low rank]\label{bias-low} Let $k$ be a finite field. For any $s>0$ and any $d \ge 1$  there exists 
$r=r(d,s,k)$ such that 
the following holds: for any $k$-vector space $V$ and any polynomial $P:V \to k$ of degree $d$ and rank $\ge r$ we have $b(P)<q^{-s}$. Moreover the function 
 $r(d,s,k)$ is independent of $k$ if $char(k)>d$.
\end{theorem}

\begin{remark}\label{low-char} a) This result 
was proved in \cite{gt} in the case when $k$ is a prime field and deg$(P)<p$, it was generalized to any prime field in \cite{kl}. The dependence on the field was clarified in \cite{BL}. \\
b) We expect  this result to hold with no restriction on the characteristic.
\end{remark}

\begin{remark}
For  $d=3,4$, an explicit bound on $r$ was worked out in \cite{hs}.
\end{remark}

\begin{corollary}\label{rank-arank} 
 For any $s>0$ and  $d \ge 2$   there exists $r=r(d, s, k)$ such that  $\|P\|_{U_2}<q^{-2s}$ for any $k$-vector 
 $V$ and a polynomial $P:V \to k$ of degree $d$ and rank $\ge r$.
 \end{corollary}

\begin{proof} Let $r_0(d,s)$ be so that Theorem \ref{bias-low} holds.  By the inverse theorem for the $U_2$ norm (see  \cite{gowers}), if $\|P\|_{U_2}>q^{-2s}$ then there is a linear polynomial $l:V \to k$ with  $b(P+l)>p^{-s}$,  thus $P+l$ is of rank $<r_0(d,s,k)$. But since $l$ is linear, $r(P) < r_{0}(d,s,k)$. 
\end{proof}

The second key ingredient  is the  existence of solutions for various systems of equations in homogeneous varieties. 
 
 \begin{proposition}\label{solutions-X} For any triple $d,m,L\geq 1$ there exists $n=n(d,m,L)$ such that for any $k$-vector space $V,$ homogeneous polynomials    $\{P_i:V\to k\}_{i=1}^L$ of degrees  $\le d$ and points  $a_j\in X,1\leq j\leq m, X:=\{x: P_i(x)=0\}$ we have $|Z|\geq q^{-n}|V|$ where 
$$Z=\{x\in V: P_i(x+a_j)=0,1\leq i\leq L,1\leq j\leq m\}.$$ 
 \end{proposition}
 
 We prove Proposition \ref{solutions-X} in the Appendix \ref{appendix-solutions}. 
 \begin{remark} As follows from  \cite{br} the analogue of this proposition is also true for local non-archimedian fields $k$.
\end{remark}

\begin{corollary} For any triple $d,m,L\geq 1$ there exists $n=n(d,m,L)$ such that for any $k$-vector space $V,$ polynomials    $\{P_i:V\to k\}_{i=1}^L$ of degrees  $\le d$ and points  $a_j\in X,1\leq j\leq m, X:=\{x: P_i(x)=0\}$ we have if $X$ is not empty then  $|Z|\geq q^{-n}|V|$
where  $$Z=\{x\in V: P_i(x+a_j)=0,1\leq i\leq L,1\leq j\leq m\}.$$ 
\end{corollary}

%%%%%%%%%%%%%%%%%%%%%%%%%%%%%%%%%%%%%%%%%%%%%%%
\section{Extending weakly linear functions - finite field case}

In this section we prove Theorem \ref{finite-multi}.  

We outline the proof. Let $k$ is a finite field and $V$ is a $k$-vector space.  Let $X\subset V$ be a  complete intersection of degree $d$ and codimension $L$ not contained in a proper linear subspace,   $X=\{P_i=0\}_{i=1}^L$. Let $f:X \to k$ be a weakly linear function.  We want to show that for $X$ of sufficiently high rank we can extend it to a linear function on $V$.  Without loss of generality we may assume that $X$ is not contained in a proper linear subspace. 

We first show that for any weakly linear function $f$ on $X$  we have  $f(x+y)=f(x)+f(y)$ for all $x,y\in X$ such that $x+y\in X$.(Lemma \ref{triples}).

 Next we show that this property  implies that $f_2(u|v,w)=0$  for almost all $u,v,w\in V$ such that  $u,u+v,u+w, u+v+w\in X$ 
This  step uses the high rank condition (Lemma \ref{parallelograms-ae}) %\footnote{At this point we could use Theorems ... in \cite{} to deduce that there is a linear function on $V$ that agrees with $f$ on $X$ almost everywhere, but instead we proceed by a different argument which is the one we later generalize when extending a quadratic function from a quadratic surface.}.

The next step is a "linear testing result"  for varieties - namely we show that if $f_2$ vanishes on almost all parallelograms in $X$ then it almost surely agrees with a function that vanishes on all parallelograms in $X$ (Proposition \ref{testing-X}). \footnote{This is a special case of a more general result on polynomial testing in  varieties in the paper \cite{kz}. We give the full argument here for completion.}; it does not require the high rank condition, but rather a general result on the abundance of solutions to equations in  varieties (Proposition \ref{solutions-X}). 

Given a function that vanishes on all parallelograms in $X$,  we  extend it to a function $h$ on 
$X+X,h(x+y):=f(x)+f(y),x,y\in X$. Since 
$f_2$ vanishes on parallelograms in $X$ the function $h$ is well defined. We show that for sufficiently high rank the function $h_2$ is defined and  vanishes on 
 a.e.  parallelograms in $V$. This implies that 
it almost surely agrees with a linear function $g$ on $V$(Proposition \ref{extension}). 

Now $g-f$ vanishes a.e. on $X$ and vanishes on all additive triples in $X$. From  this we can conclude that 
$(g-f)|_X=0$.  

\ \\
%%%%%%%%%%%%%%%%%%%%%%%%%%%%%%%%%%%%%%%%%%%%%%%%%
 \subsection*{$\quad$ Counting Lemmas}\label{counting}
 
We start with Lemmas estimating the sizes of various sets that play an important role in Theorem \ref{finite-multi}.  

\begin{definition}[$Y_2(X)$]
 Given a subset $X$ of a $k$-vector space $V$ we define 
\[
 Y_2(X)=\{(x,v_1,v_2)\in X\times V^2| (x|v_1, v_2) \in C_2(X)\}.
\] 
\end{definition}

\begin{lemma}\label{Y-large}  For any $s>0$ there exists $r=r(d, s, L,k)$ such that for any $k$-vector space $V$ and  a complete intersection  $X\subset V$  of codimension $L$, degree $d$ and  rank $> r$, not contained in a proper linear  subspace we have  
$||Y_2(X)|-q^{-4L}|V|^3|\leq q^{-s}|V|^3$. 
\end{lemma}

\begin{proof}
We count the number elements in $Y_2(X)$:
\[
q^{-4L}\sum_{\bar a,\bar b,\bar c, \bar d \in k}\sum_{x, v_1, v_2 \in V}   e_q( \sum_i a_iQ_i(x)-b_iQ_i(x+v_1)-c_iQ_i(x+v_2)+d_iQ_i(x+v_1+v_2) ).
\]
For any fixed $(\bar a, \bar b, \bar c , \bar d) \ne \bar 0$ we have 
\[
|\mE_{x, v_1, v_2 \in V}   e_q( \sum_i a_iQ_i(x)-b_iQ_i(x+v_1)-c_iQ_i(x+v_2)+d_iQ_i(x+v_1+v_2) )| \le  \min_{\bar a \neq 0}\| \sum_i a_iQ_i(x)\|_{U_2}.
\]
Since $X$ is not contained in a proper subspace, if  $\sum_i a_iQ_i(x)$ is not constant then it is of degree $\ge 2$ for any $\bar a \neq \bar 0$.
By Corollary \ref{rank-arank} for any $s>0$ we can choose $r=r(d,L,k)$ (or $r(d,L)$ when $p>d$), such that for $X$ of rank $>r$ we have $ \min_{\bar a \neq \bar 0} \|\sum_ia_iQ_i(x)\|_{U_2}<p^{-s-4L}$,  for any $\bar a \neq \bar 0$. It follows that $|Y_2(X)|=(q^{-4L}+q^{-s})|V|^3$. 
\end{proof}

\begin{remark}
 Let $\bar Y\subset \mP (V)$ be the corresponding projective variety. The previous Lemma suggests that the cohomology $H^{d-i}(\bar Y,\bar \mQ _l), i\leq s$ vanishes for odd $i$ and are of dimension $1$ for even $i$.
\end{remark}

\begin{definition}[F] Given a subset $X$ of a $V$ we define a subset $F=F_X$ in $V^2$ as the set of pairs $(v,v')$ such there is no  $x\in X$ with  
$x+v, x+v' \in X$.
\end{definition}

\begin{remark}
If $X\subset V$ is an algebraic subvariety then $F\subset V\times V$ a constructible subset.
\end{remark}

 \begin{proposition}\label{F-large-X}  
For any $s>0$ there exists $r=r(d, s, L,k)$ such that for any $k$-vector space $V$ and  a complete intersection 
 $X\subset V$  of codimension $L$, degree $d$ and rank $> r$, not contained in a proper linear subspace we have 
 $|F(k)|/|q^{2dim(V )}| \le q^{-s}$.   
\end{proposition}

 \begin{proof} 
For fixed $v,v'$ we count the  number of $y$ so that 
$y, y+v, y+v' \in X$. This is given by
\[
q^{-3L}\sum_{\bar a,\bar c,\bar c' \in k}\sum_y   e_q( \sum_i a_iQ_i(y)-c_iQ_i(y+v)-c_i'Q_i(y+v') ).
\]
Consider the average
\[
\mE_{v,v'}|q^{-3L}\sum_{\bar a,\bar c,\bar c' \in k}\mE_y   e_q( \sum_i a_iQ_i(y)-c_iQ_i(y+v)-c_i'Q_i(y+v') ) - q^{-3L}|.
\]
This is equal to 
\[
\mE_{v,v'}|q^{-3L}\sum_{(\bar a,\bar c,\bar c') \neq \bar 0}(\mE_y   e_q( \sum_i a_iQ_i(y)-c_iQ_i(y+v)-c_i'Q_i(y+v') )|,
\]
which by the triangle inequality is 
\[
\le 
q^{-3L}\sum_{(\bar a,\bar c,\bar  c' )\neq \bar 0}\mE_{v,v'}|\mE_y   e_q( \sum_i a_iQ_i(y)-c_iQ_i(y+v)-c_i'Q_i(y+v') ) |.
\]
Fix $(\bar a,\bar c,\bar  c' )$ with $\bar c' \neq \bar 0$. By the Cauchy-Schwarz inequality the inner sum squared  is bounded by
\[\begin{aligned}
&\mE_{v,v'}|\mE_y   e_q( \sum_i a_iQ_i(y)-c_iQ_i(y+v)-c_i'Q_i(y+v') ) |^2 \\
&= \mE_{v,v',y, y'}   e_q( \sum_i a_i(Q_i(y)-Q_i(y'))-c_i(Q_i(y+v)-Q_i(y'+v))-c_i'(Q_i(y+v')-Q_i (y'+v')) ) \\
&\le (\mE_{v,y, y'}  | \mE_{v'}e_q( \sum_i c_i'(Q_i(y+v')-Q_i (y'+v')) )|^2)^{1/2} = \|\sum_i c_i'Q_i(y)\|^2_{U_2}
\end{aligned}\]
Similarly when $\bar c' $ or $\bar a$ are $\neq \bar 0$. Thus we get that for $(\bar a,\bar c,\bar  c' ) \neq \bar 0$ 
\[
\mE_{v,v'}|\mE_y   e_q( \sum_i a_iQ_i(y)-c_iQ_i(y+v)-c_i'Q_i(y+v') ) | \le \min_{\bar a \neq \bar 0} \| \sum_ia_iQ_i(y)\|_{U_2}
\]
Since $X$ is not contained in a proper subspace, if  $\sum_i a_iQ_i(x)$ is not constant  then it is of degree $\ge 2$.
By Corollary \ref{rank-arank} for any $s>0$ we can choose $r=r(d,L,k)$ (or $r(d,L)$ when $p>d$), such that for $X$ of rank $>r$ we have $ \min_{\bar a \neq \bar 0} \| \sum_ia_iQ_i(y)\|_{U_2}<p^{-2s-3L}$,  for any $\bar a \neq \bar 0$.  Thus the total contribution summing over all  $(\bar a,\bar b,\bar  b' )\neq \bar 0$ is bounded by $p^{-2s}$. It follows that for 
$p^{-s}$ a.e $(v,v') \in V^2$ we have  
\[
q^{-3L}\sum_{\bar a,\bar c,\bar c' \in k}\mE_y   e_q( \sum_i a_iQ_i(y)-c_iQ_i(y+v)-c_i'Q_i(y+v') ) = q^{-3L}+O(q^{-s}).
\]
 \end{proof}

\begin{definition}[E] Given a subset $X$ of a $k$-vector space $V$ we define 
\[
E=E_X: \{(v,v', v'') \in V^3: \text{ there does not exists a $y\in X$ such that  $y+v, y+v+v',-y+v'' \in X$}\}.
\]
\end{definition}
\begin{remark}
If $X\subset V$ is an algebraic subvariety then $E\subset V\times V$ a constructible subset.
\end{remark}
 \begin{proposition}\label{E-large-X}  

For any $s>0$ there exists $r=r(d, s, L,k)$ such that for any $k$-vector space $V$ and  a complete intersection 
 $X\subset V$  of codimension $L$, degree $d$ and rank $> r$, not contained in a proper linear subspace we have $|E(k)|/|q^{3dim(V )}| \le q^{-s}$.  
 Moreover we can find a function $r(d,s,L)$ such that 
 $|E(k)|/|q^{3dim(V )}| \le q^{-s}$  if rank of $X\geq r(d,s,L)$
for all fields $k$ of characteristic $>d$.

\end{proposition}

\begin{proof} Same argument as Proposition \ref{F-large-X}.
\end{proof}

\ \\
%%%%%%%%%%%%%%%%%%%%%%%%%%%%%%%%%%%%%%%%%%%%%%%%
 
\subsection*{$\quad$ Reduction to functions vanishing on parallelograms in $X$}
Let $f:X \to k$ be a weakly linear function. In the following sequence of lemmas we show that for $X$ of sufficiently high rank  there exists 
a function $h$ that agree with $f$ a.e. and vanishes on all parallelograms in $X$. 

We first show that $f$ vanishes on additive triples in $X$  (parallelograms through the origin). 
\begin{lemma}\label{triples} Let $f:X \to k$ be weakly linear. Then $f_2(0|z,x-z)=0$ for any $x, z \in X$ such that $x-z \in X$. 
\end{lemma}
\begin{proof}
 Let $x,z, x-z \in X$.  By proposition \ref{solutions-X} there  exists $y$ such that 
\[
\langle x, y \rangle , \langle z, y \rangle, \langle y+x, y+z\rangle \in \mathcal L.
\]
Consider the following triples
\[
(x,y,x+y), (z,y,z+y), (x-z, x+y, z+y).
\]
These are all in isotropic subspaces, thus
\[
f(x+y)=f(x)+f(y), \ f(z+y)=f(z)+f(y), \  f(x+y)= f(x-z)+f(z+y)
\]
so that
$f(x)=f(z)+f(x-z). $
\end{proof}

Next we show that $f$ vanishes on parallelograms in $X$ that have a generator in $X$. 
\begin{lemma}\label{parallelograms} Let $f:X \to k$ be weakly linear. Then $f_2(x|v_1, v_2)=0$ for any $(x|v_1, v_2)\in C_2(X)$ such that $v_1 \in X$. 
\end{lemma}

\begin{proof} Let $(x|v_1, v_2)$ be a parallelogram in $X$. If $v_1 \in X$ then 
\[
(x|v_1, v_2)=(0|x+v_2, v_1)-(0|x, v_1)
\]
a difference of parallelograms through the origin, so by Lemma \ref{triples} $f_2(x|v_1, v_2)=0$. 
\end{proof}

Unfortunately we cannot write all parallelograms in $X$ as a sum of parallelograms with a generator in $X$, but we can do this for almost all parallelograms.  

\begin{lemma}\label{parallelograms-ae} 
Let $k$ be a finite field. For any $s>0$ there exists $r=r(d, s, L,k)$ such that for any $k$-vector space $V$ and subvariety  $X\subset V$  which is a complete intersection of codimension $L$, degree $d$, rank $> r$, and not contained in a proper subspace the following holds: if $f:X \to k$ is weakly linear, then $f_2(c)=0$ for  $p^{-s}$ a.e. $c \in C_2(X)$.
\end{lemma}
\begin{proof}
Let $(x|v_1, v_2)$ be a parallelogram in $X$. By Lemma \ref{Y-large}, Proposition  \ref{E-large-X} and Lemma \ref{density-ae} there exists $r=r(s,d,L,k)$ such that if $X$ is of rank $>r$, then  for $q^{-s}$ a.e $(x, v_1, v_2) \in Y$ such that   there exists $w \in X$ such that $w+x, w+x+v_2, -w+v_1 \in X$.
 Now
\[
(x|v_1, v_2)=(x+w|v_1-w, v_2)-(x|w, v_2)
\]  
a difference of two parallelograms with a generator in $X$, so by Lemma \ref{parallelograms} $f_2(x|v_1, v_2)=0$. 
\end{proof}

\ \\

%%%%%%%%%%%%%%%%%%%%%%%%%%%%%%%%%%%%%%%%%%%%%%%%%%%%%%%%%
\subsection*{$\quad$ Linear testing on algebraic varieties}\label{testing}

Below is a  general testing result  about functions from $X \to W$ where $W$ is an abelian group, which we use later in the case when $W=k$.  This result is of independent interest, and it does not require $X$ to be of high rank. A generalization of this result appears in the paper \cite{kz}.

\begin{proposition}[Testing on $X$]\label{testing-X} Let $d, L>0$. There exists an $\alpha, \beta>0$ depending on $d,L$, such that the following holds: for any   complete intersection $X$ in a vector space $V$ of degree $d$ and codimension $L$ and  any function $f:X\to W$ such that 
$f_2$ vanishes $\epsilon$-a.e on $C_2(X)$ where  $\epsilon < q^{-\alpha}$  there exists a function $h: X \to W$ such that $h_2|_{C_2(X)}\equiv  0$ and $h(x)=f(x)$ for 
$q^{\beta} \epsilon$ a.e $x \in X$. 
\end{proposition}

\begin{proof} 
Let $X$ be a variety of degree $d$ and codimension $L$ that is a complete intersection and  $f:X \to W$ be a function such that the function  $f_2$ on $ C_2(X)$  vanishes $\epsilon$-a.e. 

For $x \in X$ define
\[
Y_x=\{y,z: x+y, x+z, x+y+z \in X\}
\]
By Proposition \ref{solutions-X}  for any $x \in X$ we have $|Y_x| =q^{-O_{d,L}(1)}|V|^2$ \footnote{For sufficiently high rank $Y_x$ is approximately of size $q^{-3L}|V|^2$, but in this section we don't have any rank assumptions.}.
Define 
\[
F_x(y,z)= f(x+y)+ f(x+y)-f(x+y+z) = f'_2(x|y,z)
\]
\begin{lemma} $F_x(y,z)$ is constant $q^{-O_{d,L}(1)}\epsilon$ a.e.
\end{lemma}

\begin{proof}
Observe that for any $x,y,z, w \in V$ we have
\[
(x|y,z)=(x|w,z)- (x-w|y+w,z). 
\] 
so that for any $w, w' \in V$ we have 
\[\begin{aligned}
(x|y,z)-(x|y',z')&= (x|w,z)- (x-w|y+w,z)- [(x|w,z')- (x+w|y'-w,z')]  \\
&=  (x|w,w')-(x+w'| w,z-w')- (x+w|y-w,z) \\
&\quad - [ (x|w,w')-(x+w'| w,z'-w') - (x+w|y'-w,z')]\\
&= (x+w'| w,z'-w') + (x+w|y'-w,z') \\
&\quad - (x+w'| w,z-w')- (x+w|y-w,z)
\end{aligned}
\]
and thus
\[
F_x(y,z)-F_x(y',z') =  f_2(x+w'| w,z'-w') + f_2(x+w|y'-w,z') - f_2(x+w'| w,z-w')- f_2(x+w|y-w,z).
\]

Consider the map: $\pi_1: Y_x^2\times V^2 \to V^4$ defined by  $$(y,z ,y',z',w,w' )\mapsto  (x+w'| w,z'-w').$$ Fix  $(s|t_1, t_2) \in C_2(X)$, and consider its preimage under $\pi$. This is the set of $(y,z, y',z',w,w')$ such that 
\[\begin{aligned}
&x+y, \ x+z, \ x+y+z, \ x+y', \ x+z', \ x+y'+z' \in X; \\
&x+w'=s, \ w=t_1, z'-w'=t_2 
\end{aligned}\]

By Proposition \ref{solutions-X}, this system is solvable and has $q^{O_{L,d}(1)}|V|^3$ many solutions (the constant independent of $(s|t_1, t_2)$, and clearly it has at most $|V|^3$ solutions. Since $f_2$ vanishes 
$\epsilon$-a.e. on $C_2(X)$ we find that $f_2(\pi_1(y,z,y',z',w,w'))$ vanishes for $q^{O_{L,d}(1)}\epsilon$ a.e $(y,z,y',z',w,w') \in Y_x^2 \times V^2$, such that $\pi_1(y,z,y',z',w,w') \in C_2(X)$ 
 
Similarly for the maps, $\pi_2, \pi_3, \pi_4: Y_x^2\times V^2 \to V^4$ defined by
\[\begin{aligned}
&\pi_2:(y,z, y',z',w,w' )\mapsto (x+w|y'-w,z') \\
&\pi_3:(y, z, y',z',w,w' )\mapsto(x+w'| w,z-w')\\
&\pi_4:(y,z ,y',z',w,w' )\mapsto(x-+w|y-w,z) 
\end{aligned}\]
So that  for $i=1, \ldots, 4$ we have  $f_2(\pi_i(y,z,y',z',w,w'))$ vanishes for $q^{O_{L,d}(1)}\epsilon$ a.e $(y,z, y',z',w,w') \in Y_x^2 \times V^2$
such that $\pi_i(y,z,y',z',w,w') \in C_2(X)$. \\

Let $S$ be the set of $(y,z,y',z',w,w') \in Y_x^2 \times V^2$ such that 
\[
(x+w'| w,z'-w'), (x+w|y'-w,z') , (x+w'| w,z-w'),(x+w|y-w,z) \in C_2(X),
\]
By Proposition \ref{solutions-X}  $|S|=q^{-O_{d,L}(1)}|V|^6$. It follows that $q^{O_{L,d}(1)}\epsilon$ a.e $(y,z,y',z',w,w') \in S$ we have 
 $f_2(\pi_i(y,z,y',z',w,w'))=0$ for $i=1, \ldots, 4$.  \\
 
 Let $S_{y,z,y',z'}$ be the set of $(w,w') \in V^2$ such that $(y,z,y',z',w,w') \in S$, i.e the set of $w,w'$ so that 
 \[
 x+w', x+w+w', w+(x+z'), w+(x+z) \in X
 \] 
which by Proposition \ref{solutions-X} is of size $q^{-O_{L,d}(1)}|V|^2$ (since $x+z, x+z' \in X$). It follows that for $q^{O_{L,d}(1)}\epsilon$ a.e. 
$(y,z), (y',z')$ we can find $w, w'$ such that  $f_2(\pi_i(y,z,y',z',w,w'))=0$, and thus $F_x(y,z)-F_x(y',z') =0$. 
\end{proof}

We define $h: X \to W$ setting $h(x)$ to be the essential value of $F_x(y,z)$, when $(y,z)$ range over  $Y_x$.

\begin{lemma} For $\epsilon$ sufficiently small in terms of $d, L$, $h_2|_{X} \equiv 0$.
\end{lemma}
\begin{proof}  
Let $a,b,c \in V$. 
Consider the following array of linear forms $\{l_{ij}\}_{i=1,2,3,4; \ j=1,2,3}$. 
\[
\begin{array}{lll}
a+x & a+y & a+x+y \\
b+x' & b+y' & b+x'+y' \\
c +x'' & c+y'' & c+x''+y'' \\
-a-b-c-x-x'-x'' &-a-b-c-y-y'-y'' & -a-b-c-x-x'-x''-y-y'-y'' 
\end{array}
\]
Consider the set
\[
B= \{\bar x, \bar y : l_{ij}(\bar x, \bar y) \in X\} 
\]
Then by Proposition \ref{solutions-X} $|B| \geq q^{O_{d,L}(1)}|V|^6$.  \\

Consider also the following maps from $V^6 \to V^3$
\[\begin{aligned}
&\pi_{i} (\bar x, \bar y) =( l_{i1}(\bar x, \bar y) ,  l_{i2}(\bar x, \bar y) , l_{i3}(\bar x, \bar y) ) \\
&p_{j} (\bar x, \bar y) =( l_{1j}(\bar x, \bar y) ,  l_{2j}(\bar x, \bar y) , l_{3j}(\bar x, \bar y), l_{4j}(\bar x, \bar y) )
\end{aligned}\]

All these maps are linear maps from linear spaces so the fibers are all of the same size: for $\pi_i$ they are of size $|V|^4$, and for 
$p_i$ of size $|V|^3$. 
Let $g_1:V^3 \to V$ be defined by $g_1(v,v',v'')=f(v)+f(v')-f(v'')-h(a)$, and similarly $g_2, g_3, g_4$.  
We apply Lemma \ref{ae-projection} for the maps $\pi_i$ and functions $g_i$ and then maps $p_j$ with the maps all equal $f_2$. 
Then we find that $q^{-O_{L,d}(1)}$ a.e. $\bar x, \bar y \in B$ we have $g_i(\pi_i(\bar x, \bar y))=0$ ang $f_2(p_i(\bar x, \bar y))=0$. Thus
 $h(a)+h(b)+h(c)-h(a+b+c)=0$. 

\end{proof}

Finally we need to show that $q^{O_{L,d}(1)}\epsilon$ a.e. $x \in X$ we have $h(x)=f(x)$.  Now $C_2(X)=\bigcup_{x \in X} Y_x$, and $f_2$ vanishes $q^{O_{L,d}(1)}\epsilon$ a.e. on $C_2(X)$ thus  by Lemma \ref{fubini}  for  $q^{O_{L,d}(1)}\epsilon$ a.e. $(u,t) \in Y_x$  we have $h(x)=F_x(u,t)$ and $q^{O_{L,d}(1)}\epsilon$ a.e.$x$, for  $(u,t) \in Y_x$  we have $f(x)=F_x(u,t)$.
\end{proof}

\ \\
%%%%%%%%%%%%%%%%%%%%%%%%%%%%%%%%%%%%%%%%%%%%%%%%%%

\subsection*{$\quad$ Constructing the extension}
 Fix $s>0$. By Lemma \ref{parallelograms-ae} there exists $r=r(s,k,d, L)$  such that if $X$ is of rank $>r$ then 
$f_2(c)=0$ for  $q^{-s}$ a.e. $c \in C_2(X)$. By Proposition \ref{testing-X} we can find $h:X \to k$ such that $h_2|_{C_2(X)}\equiv 0$ and 
 $h=f$ for $q^{O_{d,L}(1)}q^{-s}$ a.e. $x \in X$. It follows that  there exists $r=r(s,k,d, L)$ such that if $X$ is of rank $>r$ then  there exists $h:X \to k$ such that $h_2|_{C_2(X)}\equiv 0$ and  $h=f$ for $q^{-s}$ a.e. $x \in X$. 
 
\begin{proposition}\label{extension} There exists $r=r(s,k,d, L)$ such that for any $X$ of degree $d$, codimension $L$ that is not contained in any proper subspace and rank $>r$,  and any $h:X \to k$ with $h_2|_{C_2(X)}\equiv 0$, $h(0)=0$ there exists a linear function  $g:V \to k$ such that  $g=h$ for $q^{-s}$-a.e. $x \in X$. 
\end{proposition}

\begin{proof}
We extend $f$ to the set  $A:=X-X$ : given $v \in A$ we choose $x,y \in X$ such that $v=x-y$ and define 

\begin{equation}\label{initial}
f(v):=f(x)-f(y). 
\end{equation} 
Observe that $f(v)$ does not depend on a choice of a pair  $x,y$. Really for any other  pair $x',y'\in X$ such that $v=x'-y'$ 
  $(y, x, y',x')$ is a parallelogram in $X$ and $f$ vanishes on it.  

We then extend $f$ to $V$ setting $f=0$ on $ A^c$.   Let $U= \bar F^c$ where $\bar F\subset V\times V$ is the closure of $F$ in the Zariski topology. Then $U$ is a Zariski-open subset of $V^2$. 

\begin{lemma}\label{2-vertices} Let $b=(x,x',x+t, x'+t)$ be a parallelogram such that $(x,x') \in U$, and $x+t,x'+t \in X$. Then $f_2(b)=0$. 
\end{lemma}

\begin{proof} 
We can find $y$ such that $y, y+x, y+x' \in X$ and thus
\[
f(x)= f(x+y)-f(y);  \quad f(x')= f(x'+y)-f(y).
\]
Also $(x+t, x'+t, x+y, x'+y)$ is a parallelogram in $X$ so 
\[
f(x+t)-f(x'+t)-f(x+y)+f(x'+y)=0.
\]
It follows that 
\[
f(x+t)-f(x'+t)-f(x)+f(x')=0.
\]
\end{proof}

\begin{lemma}\label{2-cube-good} Let $b=(x_1, x_2, x'_1, x'_2)$ be a parallelogram such that $(x_1,x_2),(x'_1,x'_2)  \in U$. Then $f_2(b)=0$. 
\end{lemma}

\begin{proof} 
We find a $y$ such that $y+x_1, y+x_2 \in X$. Now  apply Lemma \ref{2-vertices} to $(x_1, x_2, y+x_1, y+x_2)$ and $(x'_1, x'_2, y+x_1, y+x_2)$ and subtract the equations. 
\end{proof}

\begin{lemma}\label{ae-cube-finite}  Let $k$ be a finite fields and  let $s>0$. There exist $r=r(d,s,L,k)$ such that if  $X$ is of rank $>r$, then $f_2$ vanishes on $q^{-s}$-a.e element in $C_2(V)$. The rank $r(d,s,L,k)$ is independent of $k$ for all fields $k$ with  $char(k)>d$.
\end{lemma}

\begin{proof}
By Proposition \ref{E-large-X} there exist $r=r(d,s,L)$ such that for any $X$ for rank $>r$ we have $|F(k)|< q^{-2s}|V|^2$.
Let 
\[
M_v=\{ (y , y+v): y \in V\} \cap  F(k); \quad B=\{v: |M_v| \ge q^{-s}|V|\}
\]
Then $F(k) = \bigcup_v M_v$, a disjoint union. Then
\[
 |F(k)| =\sum_v|M_v| \ge q^{-s} |V||B|
\]
Since $|F(k)|\le q^{-2s}|V|^2$ then $|B| \le q^{-s} |V|$, so that for $ q^{-s}$-a.e. $v$, we have $|M_v| \le q^{-s}|V|$.  Thus for  $q^{-s}$-a.e. $v$ for $q^{-s}$-a.e. $y$ 
we have $(y , y+v) \in U(k)$.  By Lemma \ref{2-cube-good} we obtain that for  $q^{-s}$-a.e. $v$ for $2q^{-s}$-a.e. $y, y'$,  we have 
\[
f(y+v)-f(y)= f(y'+v)-f(y')
\]
Namely $f$ vanishes on $3q^{-s}$-a.e element in $C_2(V)$. 
\end{proof}

{\em Proof of Theorem  \ref{finite-multi}}. We constructed an linear function $g$ that agrees $q^{-s}$ a.e with $f$ on $X$. 
Consider the function $f-g$. It is  weakly linear and vanishes 
 $q^{-s}$ a.e on $X$.  Fix $x \in X$. By Proposition \ref{solutions-X} there are $q^{-O_{d,L}(1)}|V|$ many $y$ such that $\langle x, y \rangle$ is an isotropic subspace.  In particular for any such $y$ we have $x, y, y-x \in X$. Since $f-g$ vanishes $q^{-s}$ a.e. , if $s$ is sufficiently large we can find $y$ such that $f-g$ vanishes on $y, x-y$ and $\langle x, y \rangle$ is an isotropic subspace. But then $(f-g)(x)=(f-g)(y)+(f-g)(x-y) = 0$. 
\end{proof}

 %%%%%%%%%%%%%%%%%%%%%%%%%%%%%%%%%%%%%%%%%%%%%%%%%%%%%%%%%%%
 \section{Extending weakly quadratic functions - finite field case }
 
 In this section we prove Theorem \ref{quadratic-extension}.  
 
We outline the proof. Let $k$ be a finite field. Let $V$ be a  $k$-vector space, and  $X\subset V$  a non degenerate  hypersurface of degree $\ge 2$ and rank $r$.  Let $f:X \to k$ be a homogeneous weakly quadratic function. 
 
 We first show that $f_3$ vanishes on all cubes in $X$ via the origin, namely that $f_3(0|x,y,z) =0$ 
for all $x,y,z\in X$ such that $(0|x,y,z)\subset X$ (Lemma \ref{triples}).

next we  show that this property 
 implies that $f_3(u|\bar v)=0$  for almost all $u;v_1,v_2, v_3 \in V$ such that  $(u|\bar v)\in C_3(X)$.
In other words, $f_3$  vanishes on almost all $3$-cubes in $X$.
This  step uses the high rank condition (Lemma \ref{parallelograms-ae}). 

The next step is a "quadratic testing result"  for varieties -  when we show that the vanishing of  $f_3$ on {\it almost all} $ 3$-cubes in $X$ implies that $f$  {\it almost surely} agrees with a homogeneous function $h$ which vanishes on all $3$-cubes in $X$ (Proposition \ref{testing-X2}). This is a special case of a more general result on polynomial testing in homogeneous varieties in \cite{kz}; it does not require the high rank condition, but rather a general result on the abundance of solutions to equations in homogeneous varieties (Proposition \ref{solutions-X}). 

The next step is to show that  any homogeneous 
 function $h$ on $X$ that vanishes on all $3$-cubes in $X$  can be extended  to a homogeneous function $g$ on  $V$ which  vanishes on all $3$-cubes in $V$. For $X$  of degree $\ge 3$ this is proved in \cite{kz}. We focus on the case where $X=\{ v\in V|Q(v)=0\}$ where $Q$ is of quadratic form of high rank. 

It is clear that the function $g$ (if it exists) is defined uniquely up to an addition of a multiple of $Q$. To define $g$ uniquely  we choose $v_0 \in V \setminus X$ with $Q(v_0)=1$, and set $g(v_0)=0$. 

Let $(u,v):=Q(u)+Q(v)-Q(u+v),u,v\in V$.

Define 
$$V_0:=\{v:(v,v_0)=0\}, X_a=\{Q(x)=a^2\} \cap V_0,a\in k$$

We first show how to extend $h$ to $X_a$. For any 
$v \in  X_a$ we have $v-av_0 \in X$ and we can 
find {\it many} $3$-cubes $(v|av_0-v, v_2,v_3)$ with all vertices other than $av_0, v$ in $X$. We show that  the function $h_3(v| av_0-v, v_2,v_3)-h(v)$ on such $3$-cubes 
is a.e. constant and define $g(v)$ to be the essential value of  $h_3(v| av_0-v, v_2,v_3)-h(v)$. In this way we define our function $g$ on $X_{sq}:=\cup _{a\in k}X_a$

Next we show that if the rank of $Q$ is sufficiently high then    $g_3$ vanishes on 
all $3$-cubes in $X_{sq}$.

To construct an extension of $g$ to $V$ we  use the possibility to write   any element $c \in k$ as a sum of $2$ squares. For any 
  $v \in V_0$ we denote by $M_v$ the set of $3$-cubes $(v|\bar v)$ with all vertivies but $v$ in $X_{sq}$.  We show that $g_3(v|\bar v) -g(v)$ is a.e constant on function on $M_v$  and set $g(v)$ to be the essential value of $g_3(v|\bar v) -g(v)$.  We then show that $g_3$ vanishes on all $3$-cubes in $V_0$  (Proposition \ref{extension}).  

Finally we show that any function $g$ is that is quadratic function on $V_0$ and $g_3$ vanishes on all cubes in $X\cap V_0$ can be extended uniquely to a quadratic function on $V$ vanishing at $v_0$
(Lemma \ref{V0V}). 

We constructed a function $g$ so that $g-f$ vanishes a.e. on $X$ and vanishes on all $3$-dimensional subspaces in $X$. From  this we can conclude that 
$(g-f)|_X=0$, using the fact that any element in $x \in X$ belongs to many $3$-dimensional isotropic subspaces.  

\ \\

%%%%%%%%%%%%%%%%%%%%%%
\subsection*{$\quad$ Counting Lemmas}
In this section we prove some Lemmas on number of points on various varieties of interest. 

Let $Q:V \to k$ be a polynomial, $X=\{ v\in V|Q(v)=0\}$.
The number of $k$-points on $X$ can be expressed as an exponential sum:
 
 \begin{lemma}\label{sol-rep} Let $\{P_i\}_{i=1}^M$ be a collection of homogeneous polynomials $P_i:V \to k$. The number of solutions to the system $P_i(x)=c_i$
 is given by 
 \[
q^{-M}\sum_{a_1, \ldots, a_M \in k } \sum_{v\in V} e_q(\sum_i a_i(P_i(v)-c_i)).
 \]
 \end{lemma}
 
 \begin{proof}
Fix $v$ so that $P_i(v) \ne c_i$ for some $i$. Then $ \sum_{a_i\in k} e_q(a_i(P_i(v)-c_i))=0$. 
\end{proof}
 
  The following Lemma is classical. 
\begin{lemma} Let $Q:V \to k$  be a quadratic form of rank $r(Q)$ then 
\[
|\mE_{x \in V} e_q(Q(x))| = q^{-r(Q)/2}.
\]
\end{lemma}

\begin{proof}
Choose coordinates in $V$ such that $Q(x_1,...,x_n)=\sum _1^ra_ix_i^2, a_i\in k^\star$.
\end{proof}

 \begin{lemma} Let $Q:V \to k$ be a  quadratic form. Then $\|Q\|_{U_2} = q^{-\Omega(r(Q))}$.
 \end{lemma}

\begin{lemma}\label{solutions}  Let $U \subset V$ be a subspace of  codimension $M$, let $Q$ be a homogeneous quadratic forms and let $u_0 \in V$. Then the equation:
\[
u \in U: Q(u_0+u)=a
\]
has $(q^{-1}+q^{-\Omega(r(\bar Q|_U))})|U|$  solutions.  
\end{lemma}

\begin{proof}
The number of solutions is given by
 \[
N:=q^{-1}\sum_{m\in k}\sum_{v \in V}1_U(v) e_q( m(Q(v+u_0)-a))= (q^{-1}+O(q^{-\Omega(r(\bar Q|_U))}))|U|.
\]
\end{proof}

 \begin{lemma}\label{solutions-X} Let  $\{P_i:V\to k\}_{i=1}^M$ be  homogeneous polynomials and let
$x_j\in X,1\leq j\leq m, X:=\{x: P_i(x)=0\}$
 we have $|Z_{\bar x}|\geq q^{-O_{m,M}(1))}|V|$
where 
$$Z_{\bar x}=\{v\in V: P_i(v+x_j)=0,1\leq i\leq M,1\leq j\leq m\}.$$ 
 \end{lemma}
 
 \begin{proof}
 This follows from the appendix 1. 
 \end{proof}
 
   \begin{lemma}\label{one2many} Let $\{P_i\}_{i=1}^M$ be a collection of homogeneous polynomials  $P_i:V \to k$ of degree $\le d$, such that for $u_i \in V$ this system of equations
 $P_i(v+u_i)=c_i$ has a solution. Then the system has $q^{-O_{M,d}(1)} |V|$ many solutions. 
 \end{lemma}
 
 \begin{proof} Let $v_0$ be a solution, i.e. $P_i(v_0+u_i)=c_i$. 
 Observe that 
 \[
 P_i(v+v_0+u_i)-P_i(v_0+u_i)= \sum_{j<d} Q_{ij}(v)
 \]
 Let $Z$ be the set of solution to $\{Q_{ij}(v)=0\}_{i,j}$. Then by Proposition \ref{solutions-X}, $|Z|=q^{-O_{M,d}(1)} |V|$, and for any $y \in Z$ we have 
$P_i(y+v_0+u_i)-P_i(v_0+u_i)=0$ so that $P_i(y+v_0+u_i)=c_i$. 
 \end{proof}

  We will need the following generalization of Proposition \ref{solutions-X}
\begin{proposition}\label{solutions-X-1}
For any $s>0$ and  $d \ge 2$   there exists $r=r(d, s, k)$ such for any $k$-vector 
 $V$ and a polynomial $Q:V \to k$ of degree $d\ge 2$ and rank $\ge r$, $ X:=\{v: Q(v)=0\}$, for $q^{-s}$ a.e. $(\bar x,y) \in X^m \times V$,
we have $|Z_{\bar x, y}|\geq q^{-O(1)-s}|V|$.
where 
$$Z_{\bar x, y}=\{v\in V: Q(v+x_j)=0, 1\leq j\leq m\} \cap \{v \in V: Q(v+y)=0\}$$ 
In the case where $Q$ is quadratic we can take $r=O(s)$. 
 \end{proposition}

\begin{proof} 
The number of points in $Z_{\bar x,y}$ is given by
\[
(*) \quad q^{-(m+1)}\sum_{\bar r,  t}\sum_{x \in V}   e_q( \sum_{j} r_{j}Q(x+x_j)+tQ(x+y) )
\]
Consider the contribution to the sum when  $ t$  is not $0$. We claim that almost surely this contribution is negligible. Indeed
\[\begin{aligned}
&\mE_{\bar x, y'}| \mE_{x \in V}   e_q( \sum_{j} r_{j}Q(x+x_j)+tQ(x+y) )|^2 \\
&\le \mE_{\bar x, y}| \mE_{x.x' \in V}   e_q( \sum_{j} r_{j}Q(x+x_j)-r_{j}Q(x'+x_j)+tQ(x'+y)-tQ(x+y) )| \\
& \le \mE_{\bar x, y, y'}| \mE_{x, x' \in V}   e_q(  tQ(x'+y)-tQ(x+y)-tQ(x'+y')-tQ(x+y') )| \\
& \le \min_{t \neq 0}\| t Q(x)\|_{U_2}
\end{aligned}\]
By Lemma \ref{rank-arank} we can choose $r$ so that for $q^{-s}$ a.e $\bar x, y$ we have that $(*)$ is $q^{-s}|V|$ close to 
\[
q^{-(m+1)}\sum_{\bar r}\sum_{x \in V}   e_q( \sum_{j} r_{j}Q(x+x_j) ).
\]
But this is now the number of points in  
$$Z_{\bar x}=\{x\in V: Q(x+x_j)=0,1\leq j\leq m\}$$
for which we can apply Proposition \ref{solutions-X} .
\end{proof}

\begin{definition}[$Y_3(X)$]
 Given a subset $X$ of a $k$-vector space $V$ we define 
\[
Y_3(X)=\{(x,v_1, v_2, v_3): (x|\bar v) \in C_3(X)\}.
\] 
\end{definition}

\begin{lemma}\label{Y3-large}  For any $k$-vector space $V$ and    $X\subset V$  a hypersurface of degree $d$ we have $|Y_3(X)| =q^{-O_d(1)} |V|^4$. 
\end{lemma}

\begin{lemma}\label{completion} For any $k$-vector space $V$, hypersurface  $X\subset V$ degree $d$  any $x \in X$ has  $q^{-O_d(1)}|V|^3$  many completions to a $3$-cube in $X$.
\end{lemma}

\begin{proof} 
By Proposition \ref{solutions-X} the system has $q^{-O_d(1)}|V|^3$ many solutions.\end{proof}

The following lemmas are special to the case where $X$ is a quadratic surface. 

 \begin{lemma}\label{quadratic-sol} Let $Q$ be a quadratic form of rank $\geq 13$. Then for any symmetric matrix $D_{ab},1\leq a,b\leq 6$ there exist vectors $v_a\in V$ such that $Q(v_a)=D_{aa}, (v_a,v_b)=D_{ab}$.
\end{lemma}
\begin{proof} 
It is clear that we can find a subspace $V'\subset V$ of dimension $13$  such that the restriction $Q_{|V'}$ is a non-degenerate. So  it is sufficient to consider the case when $Q$ is a non-degenerate for on a $13$-dimensional space.

The matrix $D_{ab},1\leq a,b\leq 6$ defines a quadratic form $q'$ on the space $L=k^6$. It is easy co construct an extension of $q'$ to a non-degenerate quadratic form $q''$ on $L\oplus L^\vee$ and then  to a non-degenerate quadratic form $q$ on $L\oplus L^\vee\oplus k$ such that $disc(q)=disc(Q)\in k^\star/(k^\star )^2$. Since any two non-degenerate quadratic forms on an odd-dimensional spaces with the same discreminant are equivalent.

\end{proof}

\begin{lemma}\label{sol-array} Let $Q$ be a quadratic form of rank $\geq 8$. 
The system of  equations
\[ \begin{aligned}
&Q( \omega \cdot ( z_1',  z_2',  z_3')) =-\alpha_{\omega} \\
&Q( \omega \cdot ( z_1'',  z_2'',  z_3'')) =-\beta_{\omega} \\
&Q( \omega \cdot (z_1' +z_1'', z_2' +z_2'', z_3' +z_3'')) =-\gamma_{\omega}; \qquad \bo \in \2^3
\end{aligned}\]
under the condition $\alpha _{\bar 0}=\beta _{\bar 0}=\Gg _{\bar 0}=0$ and
$$\sum _\bo \alpha _\bo =\sum _\bo \beta _\bo =\sum _\bo \Gg _\bo=0$$
has $q^{O(1)}|V|^6$ many solutions.  
\end{lemma}

\begin{proof}
By Lemma \ref{one2many} it suffices to show that the system is solvable. 
Let $$a_{ij}=(z'_i,z'_j), \ b_{ij}=(z'_i,z''_j)+(z'_j ,z''_i), \ c_{ij}=(z''_i,z''_j), \ i<j,$$
and
 $$A_i=Q(z'_i), \ C_i=Q(z''_i), \ B_i=(z'_i,z''_i).$$ 
\ \\
The first two equations define $a_{ij}, A_i,c_{ij},C_i.$  
The third system of equation for 
$\bo = (1,0,0),(0 1 0),(0 0 1)$ define $B_i$.
The last  $3$ equations define $b_{ij},i<j$. \\
\ \\
By Lemma \ref{quadratic-sol} we can find $z_i',z_i''$ such that 

 $a_{ij}=(z'_i,z'_j),b_{ij}=(z'_i,z''_j)+(z'_j z''_i),c_{ij}=(z''_i,z''_j),i<j$

$ A_i=Q(z'_i), C_i=Q(z''_i), B_i=(z'_i,z''_i)$ 
\end{proof}

In particular we get:

\begin{lemma}\label{equi} For any $\bar a \in k^8$ such that $\sum -1^{|\bo|}a_{\omega}=0$ then there are  $O(q^{-O(1)})|V|^4 $ many  $3$-cubes 
with $Q(u+\bo \cdot  \bar u)=a_{\bo}$. 
\end{lemma}

\begin{lemma}\label{opposite} Let $Q$ be quadratic and let $(u|u_1, u_2)$ be a non degenerate square and let $Q(u+\omega \cdot  \bar u)=a_{\bo}$. Let $b=a_{00}-a_{01}-a_{10}+a_{11}$. 
There exists  $q^{-O(1)}|V|$ many values of $y$ s such that $Q(y+u)=b$ and $Q(y+u + \bo \cdot \bar u)=0$ for $\bo \neq \bar 0$. 
\end{lemma}

\begin{proof}
We seek $y$ such that
\[\begin{aligned}
&Q(y)+(y,u)=-a_{01}-a_{10}+a_{11}, \quad Q(y)+(y,u+u_1)= -a_{10},\\
& Q(y)+(y,u+u_1)= -a_{01}, \quad Q(y)+(y,u+u_1+u_2)= -a_{11}.
\end{aligned}\]
Namely 
\[
(y,u_1)= a_{01}-a_{11}, \quad (y,u_2)= a_{10}-a_{11}, \quad Q(y)+(y,u)=-a_{01}-a_{10}+a_{11}
\]
This system is solvable by Lemma \ref{solutions}.
\end{proof}

\begin{lemma}\label{opposite1} Let $(u|u_1, u_2)$ be a non degenerate square and let $Q(u+\omega \cdot  \bar u)=a_{\bo}$. Let $b=a_{00}-a_{01}-a_{10}+a_{11}$. 
Let $t,s$ be with $t-s=b$. There exists  $q^{-O(1)}|V|$ many values of $y$ such that $Q(y+u)=t$, $Q(y+u+u_1)=s$ and $Q(y+u + u_2)=Q(y+u + u_1+u_2)=0$ . 
\end{lemma}

\ \\

%%%%%%%%%%%%%%%%%%%%%%%%%%%%%%%%%%%%%%%%%%%%%%%%%%%%%%%%
 
\subsection*{$\quad$ Reduction to  homogeneous functions vanishing on cubes in $X$}

%%%%%%%%%%%%%%%%%%%%%%%%%%%%%%%%%
  
Let $k$ be a finite field. Let $V$ be a  $k$-vector space, and  $X\subset V$  a non degenerate  hypersurface of degree $\ge 2$.  Let $f:X \to k$ be a homogeneous weakly quadratic function. 

First we show that we can assume that $f$ is homogeneous:
\begin{lemma}
It is sufficient to prove Theorem \ref{quadratic-extension} in the case where $f:X \to k$ is  homogeneous. 
\end{lemma}

\begin{proof}
Let $f:X \to k$  be weakly quadratic.  Subtracting $f(\bar 0)$ we may assume that $f(\bar 0)=0$. 
Write $f=f'+f''$ with  $f'(-v) = -f'(v)$ ,$f''(-v) = f''(v)$. It is clear that $f'$,$f''$ are weakly quadratic if $f$ is weakly quadratic. Let $L \subset V$ be an isotropic subspace (so $L \subset X$). Then the restriction $f|_L$ of $f$ on $L$ is quadratic.  
Since 
$f'(-v) = -f'(v)$ we see that $f'|_L$ is linear, so that $f'$ is weakly linear and by Theorem \ref{finite-multi} it can be extended to a linear function on $V$. So we can assume that  $f(-v) = f(v)$. In this case $f|_L$ is a homogeneous quadratic form on any isotropic subspace $L$ of $V$.
\end{proof}

  Let $f:X \to k$ be a homogeneous weakly quadratic function. We show that for $X$ of sufficiently high rank  there exists  a function $h$ that agrees with $f$ a.e. and vanishes on all $3$-cubes in $X$. 

\begin{proposition}\label{1reduction}
For any $s>0$  There exist $r=r(s,d,k)>0$  such that the following holds: 
For any finite field $k$,   any  $k$-vector space $V$,   any degree $d$ hypersurface $X$ of $V$  of rank $\geq r$ and any  weakly quadratic homogeneous function  $f:X \to k$ there exists 
a homogeneous function $h$ on $X$ such that $h_3$ vanishes on all $3$-cubes in $X$ and 
$h(x)=f(x)$ for   $q^{-s}$ a.e. $x \in X$. If $d=2$ we can take $r=O(s)$.
\end{proposition}

\begin{proof}
 Let  $f:X \to k$ that be homogeneous and weakly quadratic.
We first show that $f$ vanishes on $3$-cubes containing $0$. 

\begin{lemma}\label{triples1}   $f_3(0|\bar x)=0$ for any $\bar x$ such that $(0|\bar x) \in C_3(X)$. 
\end{lemma}
\begin{proof}
 Let  $\bar x$ such that $(0|\bar x) \in C_3(X)$.   Note that if $\langle v_1, v_2, v_3 \rangle$ is an isotropic subspace then in particular $(0| \bar v) \in C_3(X)$, and since $f$ is weakly quadratic $f_3(0| \bar v)=0$.  
For $\bar u^i \in V^2$, $i=1,2,3$, consider the array 
\[
(\omega \cdot ( \nu \cdot \bar x , \nu \cdot (\bar u^1, \bar u^2,  \bar u^3))_{\omega \in \2^3 \setminus \bar 0,\nu \in \2^3 \setminus \bar 0}
\]
 If we can find $\bar u^1 , \bar u^2 , \bar u^3$ so that  for any $\nu \in \2^3 \setminus \bar 0$
 \[
 (\omega \cdot ( \nu \cdot \bar x , \nu \cdot (\bar u^1, \bar u^2,  \bar u^3))_{\omega \in \2^3 \setminus \bar 0}
 \]
 lie in an isotropic subspace and for any $\omega \in \2^3 \setminus \{ \bar 0,(1,0,0)\}$
 \[
 (\omega \cdot ( \nu \cdot \bar x , \nu \cdot (\bar u^1, \bar u^2,  \bar u^3))_{\nu \in \2^3 \setminus \bar 0}
 \]
 also lie in an isotropic subspace then 
 \[
 f_3(0, ((1,0,0) \cdot ( \nu \cdot \bar x , \nu \cdot (\bar u^1, \bar u^2,  \bar u^3))_{\nu \in \2^3 \setminus \bar 0}) = f_3(0 |\bar x) =0.
 \]
now the existence of such $\bar u^1 , \bar u^2 , \bar u^3$ follows from Proposition \ref{solutions-X}.
 \end{proof}

Next we show that $f$ vanishes on $3$-cubes in $X$ that have two generators in $X$. 

\begin{lemma}\label{parallelograms3}  $f_3(x|\bar v)=0$ for any $(x|\bar v)\in C_3(X)$ such that $v_1, v_2 \in X$. 
\end{lemma}

\begin{proof} Let $(x|\bar v) \in C_3(X)$.  If $v_1, v_2 \in X$ then 
\[
(x|v_1, v_2, v_3)=(0|v_1, v_2, x+v_3)-(0|v_1,v_2, x)
\]
a difference of $3$-cubes through the origin, so by Lemma \ref{triples1} $f_3(x|v_1, v_2, v_3)=0$. 
\end{proof}

 Unfortunately we can not write all $3$-cubes in $X$ as sums of cubes with $2$  generators in $X$ but this is {\it almost} possible.

 \begin{lemma}\label{parallelograms3-ae}  For $r$ sufficiently large depending on $s,k,d$ we have
 $f_3(c)=0$ for  $q^{-s}$ a.e. $c \in C_3(X)$ with one generator in $X$. 
\end{lemma}
\begin{proof}
Let $(x|v_1, v_2, v_3)$ be a $3$-cube in $X$, with $v_3 \in X$. By Lemma \ref{Y3-large}, Proposition  \ref{solutions-X-1} and Lemma \ref{density-ae}  for $p^{-\Omega(r)}$ a.e $(x, v_1, v_2, v_3) \in Y$ such that  $v_3 \in X$ there exists $w \in X$ such that 
\[
w+x, w+x+v_2,x+w+v_3, x+w+v_2+v_3, -w+v_1 \in X.
\]
 Now
\[
(x|v_1, v_2, v_3)=(x+w|v_1-w, v_2, v_3)-(x|w, v_2, v_3)
\]  
a difference of two cubes with two generators in $X$, so by Lemma \ref{parallelograms3} $f_3(x|v_1, v_2, v_3)=0$. 
\end{proof}

Finally almost any $3$-cube in $X$ is a sum of two $3$-cubes with a generator in $X$: 

\begin{lemma}\label{parallelograms-ae-1} For $r$ sufficiently large depending on $s,k,d$ we have
 $f_3(c)=0$ for  $p^{-s}$ a.e. $c \in C_3(X)$ 
\end{lemma}

\begin{proof}
Let $(x|v_1, v_2, v_3)$ be a $3$-cube in $X$.  As in previous Lemma there exists $w \in X$ such that 
\[
w+x, w+x+v_2,x+w+v_3, x+w+v_2+v_3, -w+v_1 \in X.
\]
 Now
\[
(x|v_1, v_2, v_3)=(x+w|v_1-w, v_2, v_3)-(x|w, v_2, v_3)
\]  
a difference of two cubes with one generators in $X$, so by Lemma \ref{parallelograms3} $f_3(x|v_1, v_2, v_3)=0$. 
\end{proof}

Finally we need a quadratic testing statement about functions from $X \to k$. We proved the linear version of this in the previous section in Proposition \ref{testing-X}

\begin{proposition}[Testing on $X$ \cite{kz}]\label{testing-X2} Let $d, L>0$. There exists an $\alpha, \beta>0$ depending on $d,L$, such that the following holds: for any  complete intersection $X$ in a vector space $V$ of degree $d$ codimension $L$ and  a function $f:X\to k$ such that 
$f_3$ vanishes $\epsilon$-a.e on $C_3(X)$ where  $\epsilon < q^{-\alpha}$  there exists a function $h: X \to W$ such that $h_3|_{C_3(X)}\equiv  0$ and $h(x)=f(x)$ for 
$q^{\beta} \epsilon$ a.e $x \in X$. 
\end{proposition}

Applying Proposition \ref{testing-X2} we obtain that for $r$ sufficiently large depending on $s,k,d$  there exists a function $h$ such that $h_3|_{C_3(X)} \equiv 0$ and $h(x)=f(x)$ for $q^{-s}$ a.e. $x \in X$ ($q^{-\Omega(r)}$ in the case where $X$ is quadratic).
\end{proof}

\ \\
%%%%%%%%%%%%%%%%%%%%%%%%%%%%%%%%%%%%%%%%%%

\subsection*{$\quad$ Constructing the extension}

Let $k$ be a finite field. Let $V$ be a  $k$-vector space, and  $X\subset V$  a  hypersurface of  rank $r$ and degree $\ge 3$.  Let $f:X \to k$ be a homogeneous weakly quadratic function. By Proposition \ref{1reduction} we constructed a homogeneous  function $h:X \to k$ such that $h(x)=f(x)$ for $q^{-s}$ a.e $x \in X$, and $h_3$ vanishes on $3$ cubes in $X$.  We now wish to extend $h$ to 
a function $g:V \to k$ vanishing on all cubes in $V$.  In the case where $d\ge 3$ this is done in \cite{kz}. 

We henceforth assume that $d=2$, and we have a homogeneous  function $h:X \to k$ such that $h(x)=f(x)$ for $q^{-\Omega(r)}$ a.e $x \in X$, and $h_3$ vanishes on $3$ cubes in $X$.

The main theorem is the following: 
 \begin{theorem}\label{extension1} There exists $r_0>0$ such that the following holds : for any finite filed $k$,  $V$ a  $k$-vector space,  $X$   a quadratic hypersurface of rank $r>r_0$, that is not contained in any proper subspace,  and any $h:X \to k$ with $h_3|_{C_3(X)}\equiv 0$, and $h(0)=0$, there exists a quadratic function  $g:V \to k$ extending $h$. 
\end{theorem}

\begin{proof} 
Let $Q$ be a quadratic form. 
Recall that
 $$(v,v'):=Q(v+v')-Q(v)-Q(v'); \quad X=\{ v\in V|Q(v)=0\}.$$

Fix $v_0 \in V \setminus X$, with $Q(v_0)=1$.
Denote $S(k)$ the set of squares in $k$. 
\[
V_0=\{x:(x,v_0)=0\} \quad  \text{and} \quad  X_{a}=\{ x: Q(x)=a^2, 1 \le i\le L\} \cap  V_0.
\]
Observe that if $v \in X_{a}$ then $v - av_0 \in X_0$. Set $g(v_0)=0$.  \\

We start by extending $h$ to 
$X_a$  for  $a \in k$.
We denote by  $Z_{a,v}$ the set
\[
Z_{a,v}=\{ (y,z) \in V_0^2: (av_0+ \omega \cdot (v-av_0,y,z))_{\omega \in \{0,1\}^3\setminus \{\bar 0, 100 \}} \in X_0\}.
\]
Then $(y,z)$ is in $Z_{a,v}$, if  $y,z$ are generators of a cube with an edge $(av_0,v)$ and all other vertices in $X_0$.  \\
\ \\
For $v \in X_a$ we define
\[\begin{aligned}
F_v(y,z) : &=  f(av_0+y)+f(av_0+z)-f(av_0+y+z)-f(v+y)-f(v+z)+f(v+y+z) \\
\end{aligned}\]

\begin{lemma} For all $v \in X_a$ we have $|Z_{a,v}| \ge q^{-O(1)}|V|^2$. 
\end{lemma}

\begin{proof} We need to count the number of $z, y$ such that 
\[
v+y, v+z, v+y+z, av_0+y,av_0+z,av_0+y+z \in X_0.
\]
It suffices to solve:
\[
Q(y)=Q(z)=(y,z)= (y, v_0)=(y, v) = (z, v_0)=(z, v)= 0.
\]
which is solvable and has many solutions by Lemma \ref{solutions-X}.
\end{proof}

Say that $z, z'$ are independent if for any $a,b \in k$ such  that $(a,b) \neq \bar 0$, the linear form $(az+bz', v)$ is not trivial. 
We have that $q^{-\Omega(r(X))}$ a.e.  $(z,z') \in V^2$ are independent.

\begin{lemma}\label{def-extension-X_a} For all $v \in X_a$ for any  $(y,z), (y',z') \in Z_{a,v}$ with $z, z'$ independent we have $F_v(y,z)=F_v(y',z')$.\end{lemma}

\begin{proof} 
Let $(y,z), (y',z') \in Z_{a,v}$. Observe that 
\[ \begin{aligned}
F_v(y,z) -F_v(y',z') &= f_3(av_0| v-av_0, y,z)-f_3(av_0| v-av_0, y',z')\\
&=  f_3(av_0| v-av_0, w,z)- f_3(v_0+w; v-av_0, y-w,z) \\
&\quad - [f_3(av_0| v-av_0, w,z')-f_3(av_0+w| v-av_0, y'-w,z')]
\end{aligned}\]
We need to show that we can find $w \in V_0$ such that 
\begin{equation}\label{1}
av_0+w, av_0+z+w, av_0+z'+w,  v+z+w, v+z'+w, v+w \in X_0
\end{equation}
since  then we have both cubes $(av_0+w| v-av_0, y-w,z), (av_0+w| v-av_0, y'-w,z')$ are cubes in $X_0$ so that
\[
 f_3(av_0+w| v-av_0, y-w,z)=f_3(av_0+w| v-av_0, y'-w,z')=0. 
\]
We are given that $$Q(av_0+z)=Q(av_0+z')=Q(v+z)=Q(v+z')=0,$$ and recall that $(x,v_0)=0$ for all $x$, thus  the conditions in \eqref{1} are satisfied if 
\[
Q(w)=-a^2; (w,z)=(w,z')=a^2, \ (v,w)=0.
\]
Let $U=\{w: (w,z)=(w,z')=a^2\}$. Then for  $(z,z')$ independent,  $U$ has many points, and thus by Lemma \ref{solutions} the above system of equations has many solutions.

We thus get for $z,z'$ independent, 
\[
F_v(y,z) -F_v(y',z')=f_3(av_0| v-av_0, w,z)-f_3(av_0| v-av_0, w,z')=0
\]
since the cube 
$$(av_0| v-av_0, w,z)-(av_0| v-av_0, w,z') = (av_0+z;  v-av_0, w, z'-z)$$ 
 is a cube in $X$. 
\end{proof}

Set $X_{sq}=\bigcup_{a \in S(k)} X_a$. 

\begin{corollary}[Definition of the extension on $X_{sq}$]\label{def} Let $v \in X_{sq}$. By Lemma \ref{def-extension-X_a} the function $F_v(y,z)$ is constant  $q^{-\Omega(r(X))}$ a.e. Define $g(v)$ to be this value. \\
\end{corollary}

\begin{proposition} $g$ vanishes on all  $3$-cubes $(u|\bar u)$ in $X_{sq}$. 
\end{proposition}

\begin{proof} Let $(u|\bar u)$ be a $3$-cube with vertices in  $X_{sq}$. Namely $u+\omega \cdot \bar u \in X_{a_{\omega}}$,  and 
\[
Q(u+\omega \cdot \bar u) = a^2_{\omega}; \quad a_{\omega} \in k, \ \omega \in \2^3.
\]
Since $( u|\bar u)$ is a $3$-cube we have $\sum_{\omega \in \2^3} -1^{|\omega|}a^2_{\omega} =0$. \\

Consider the array in the variables $z', z''$, $\bar z', \bar z'' \in V^3$:
\[ \begin{aligned}
(& u+\omega \cdot \bar u, \  u+ z'+\omega \cdot (\bar u+\bar z'), \ u+ z''+\omega \cdot (\bar u+\bar z''),  \ u+ z'+z''+ \omega \cdot (\bar u+\bar z'+\bar z''),  \\
& a_{\bo}v_0, \ 
 a_{\bo}v_0+z'+\omega \cdot \bar z',\    a_{\bo}v_0+z''+\omega \cdot \bar z'', \ a_{\bo}v_0+z'+z''+\omega \cdot(\bar z'+\bar z'')
)_{\bo \in \2^3 \setminus \bar 0}
\end{aligned} \]
 For any fixed $\bo$ this is a $3$-cube with an edge $(u+ \omega \cdot \bar u, a_{\bo} v_0)$. 
 We want to find many $z', z'', \bar z', \bar z ''$ so that all vertices other than $u+\omega \cdot \bar u, a_{\bo} v_0$ are in $X_0$.  For this we need to solve the system of equations:
 \begin{equation}\label{bla1}\begin{aligned}
& Q( z'+\omega \cdot \bar z')=  Q(z''+ \omega \cdot \bar z'')= - Q( z'+z''+\omega \cdot (\bar z'+\bar z''))=-a_{\bo}  \\
&( u+\omega \cdot \bar u, z'+\omega \cdot \bar z')=(u+ \omega \cdot \bar u,z'+ \omega \cdot \bar z'')=(  v_0, z'+\omega \cdot \bar z')=(v_0, z''+\omega \cdot \bar z'')=0.
 \end{aligned} \end{equation}
 
 By  Lemma  \ref{sol-array} the system \eqref{bla1} has a solution, and by Lemma \ref{one2many} it has $q^{-O(1)}|V|^8$ many solutions; denote $B$ the set of solutions. 
 
For $\bo \in \2^3\setminus \bar 0$, consider the maps $\pi_{\bo}: V^8 \to V^8$
\[\begin{aligned}
(z', z'', \bar z', \bar z'') \mapsto (&u+z'+\omega \cdot (\bar u+\bar z'), \ u+z''+ \omega \cdot (\bar u+\bar z''),  \ u+z'+z''+ \omega \cdot (\bar u+\bar z'+\bar z''),  \\
& a_{\bo}v_0+z'+\omega \cdot \bar z',\    a_{\bo}v_0+z''+\omega \cdot \bar z'', \ a_{\bo}v_0+z'+z''+\omega \cdot(\bar z'+\bar z'')
)
\end{aligned}\]
Let $p$ be a point in the image. The sieve $\pi^{-1}_{\bo}(p)$ is of size $\ge q^{-O(1)}|V|^8$.  
Observe that $B \subset \pi^{-1}_{\bo}(X_0^6)$. 
By Corollary \ref{def}, and Lemma \ref{ae-projection} It follows that 
$q^{-\Omega(r)}$ a.e. $z', z'' \in V, \bar z', \bar z'' \in V^3$ such that  $\pi_{\bo}(z', z'', \bar z', \bar z'' ) \in X_0^6$ we have 
\[
g(u+\bo \cdot \bar u) = F_{u+\bo \cdot \bar u}(z'+\bo \cdot \bar z',z''+ \bo \cdot \bar z'').
\]

Similarly for $\nu, \nu' \in \2^2\setminus \bar 0$ consider the maps $\tau_{\nu}, \sigma_{\nu'}: V^6 \to V^6$
\[\begin{aligned}
&\tau_{\nu}:(\bar z', \bar z'') \mapsto (u+\omega \cdot \bar u+\nu \cdot (z'+\omega \cdot \bar z', z''+\omega \cdot \bar z''))_{\bo \in \2^3 \setminus \bar 0},  \\
&\sigma_{\nu}:(\bar z', \bar z'') \mapsto (a_{\bo}v_0+\nu' \cdot (z'+\omega \cdot \bar z', z''+\omega \cdot \bar z''))_{\bo \in \2^3 \setminus \bar 0},
\end{aligned}\]

The sieve over any point $p$ in the image $\tau^{-1}_{\nu}(p)$ is of size $\ge q^{-O(1)}|V|^8$. It
follows that 
$q^{-\Omega(r)}$ a.e. $z', z'' \in V, \bar z', \bar z'' \in V^3$ such that  $\tau_{\nu}(z',z'', \bar z', \bar z'' ) \in X_0^7$ we have 
\[
g_3((u+\omega \cdot \bar u+\nu \cdot (z'+\omega \cdot \bar z', z''+\omega \cdot \bar z''))_{\bo \in\2^3} ) =0,
\]
and similarly for $\sigma_{\nu'}$.

  In particular there exists a point $(z',z'', \bar z',  \bar z'') \in V^8$ so that for all $\bo \in \2^3 \setminus \bar 0$
\[
g(u+\bo \cdot \bar  u) = F_{u+\bo \cdot \bar u}(z'+\bo \cdot \bar z',z''+ \bo \cdot \bar z''), 
\]
and for all $\nu, \nu' \in \2^2\setminus \bar 0$ 
\[
g_3((u+\omega \cdot \bar u+\nu \cdot (z'+\omega \cdot \bar z', z''+\omega \cdot \bar z''))_{\bo \in\2^3} ) 
= g_3((a_{\bo}v_0+\nu \cdot (z'+\omega \cdot \bar z', z''+\omega \cdot \bar z''))_{\bo \in\2^3} ) 
=0.
\]
But this implies that $g_3(u| \bar u  )=0$.  
 \end{proof}

We now have a function $g:X_{sq} \to k$, with $g_3$ vanishes on all cubes through in $X_{sq}$. We want 
to extend $g$  to a quadratic  function on $V_0$. \\

The following result is well known

\begin{claim}\label{sumofsquares} Let $k$ be a finite field. Then any element in $k$ can be written as a sum of two squares. 
\end{claim}

Recall that $V_0=\{v:(v, v_0)=0\}$. Given $v\in V_0$ we want to place it in a cube with all other vertices in $X_{sq}$:

\begin{lemma}\label{complete1} Let $Q$ be a quadratic form such that  $r(Q|_{V_0})>15$. Let $v \in V_0$, with $Q(v)=s$, and let $a,b \in k$ be with $a+b=s$. Then there exists $v_1, v_2, v_3 \in V_0$ such that 
$Q(v+v_1)=a, Q(v+v_2)=b$ and all other vertices of $(v|\bar v) \in X_0$.  
\end{lemma} 

\begin{proof} We need $v_1, v_2, v_3 \in V_0$ such that $Q(v+v_1)=a, Q(v+v_2)=b$ and 
\[
Q(v+v_1+v_2)=Q(v+v_3)=Q(v+v_1+v_3)=Q(v+v_2+v_3) =Q(v+v_1+v_2+v_3)= 0.
\]
This amounts to finding $v_1, v_2, v_3 \in V_0$  such that
\[\begin{aligned}
&Q(v_1)=a-s=-b, Q(v_2)=b-s=-a, (v,v_1)=(v,v_2)=0, (v_1, v_2)=s-a-b=0, \\
&Q(v_3)=-s, (v_3, v)=0, (v_2, v_3)=s-b=a, (v_1, v_3) = -b.
\end{aligned}\]
which is possible if $r(Q|_{V_0})>15$.
\end{proof}

Let $v \in V_0\setminus X_{sq}$, then by Lemma \ref{sumofsquares} and Lemma \ref{complete1} and Lemma \ref{one2many} there are $q^{-O(1)} |V|^2$ cubes $(v|\bar v)$ such that all vertices but $v$ are in $X_{sq}$. Define
\[
G_v(\bar v) : = g_3(v|\bar v)-g(v).
\]
\begin{lemma} $F_v(\bar v)$ is constant on $q^{-\Omega(r)}$-a.e on cubes for which  all vertices but $v$ are  in $X_{sq}$.
\end{lemma}

\begin{proof} 
Suppose $(v| \bar v)$, $(v| \bar v')$ are  cubes with all vertices but $v$ are in $X_{sq}$.
 If $v_1=v_1'$, and $v_2=v_2'$ then  $(v|\bar v), (v \bar v')$ share a face so that 
 $(v+v_3|v_3'-v_3, v_1, v_2)$ is a cube in $X_{sq}$, thus $g_3$ vanishes on it, and we get $G_v(\bar v)=G_v(\bar v')$.

Suppose now that  If $v_2=v'_2$ then  $(v|\bar v), (v \bar v')$ share an edge. We find $y$ such that  $(v+y|v_1, v_2), (v+y|v'_1, v_2) \in X_{sq}$. For this we need to find $y$ with  
\[
y+v, y+v+v_1, y+v+v_1',  y+v+v_2,  y+v+v_1+v_2, y+v+v_1'+v_2 \in X_{sq}.
\]
Let $z=y+v$, we seek $y'$ such that
\[
z, z+v_1, z+v_1',  z+v_2,  z+v_1+v_2, z+v_1'+v_2 \in X_{sq}.
\]
Let 
\[
\alpha=Q(v_1)+Q(v_2)-Q(v_1+v_2),  \alpha'=Q(v'_1)+Q(v_2)-Q(v'_1+v_2).
\]
Find $t_1, t_2, t'_1, t_2' \in k^2$ so that  
\[
t_2-t_1=\alpha, t'_2-t'_1=\alpha'.
\]
We solve: 
\[
Q(z)=0, Q(z+v_2)=0, Q(z+v_1)=t_1, Q(z+v_1')=t_1',  Q(z+v_1+v_2)=t_2, Q(z+v_1'+v_2)=t_2',
\]
which is the same as 
\[\begin{aligned}
&Q(z)=0, (z,v_2)=-Q(v_2),  (z,v_1)=-t_1-Q(v_1),  (z,v_1+v_2)=-t_2-Q(v_1+v_2),  \\
& (z,v'_1)=-t'_1-Q(v'_1),  (y,v'_1+v_2)=-t'_2-Q(v'_1+v_2), 
\end{aligned}\]
So really we have
\[\begin{aligned}
&Q(z)=0, (z,v_2)=-Q(v_2),  (z,v_1)=-t_1-Q(v_1), (z,v'_1)=-t'_1-Q(v'_1), 
\end{aligned}\]
which is solvable so long as $v_1, v_2, v_1'$ are independent. 

 We then have 
 \[
  G_v(v_1, v_2, v_3)= G_v(v_1, v_2, y)=  G_v(v'_1, v_2, y)=G_v(v'_1, v_2, v_3).
 \]
 
 Finally if $(v| \bar v)$, $(v| \bar v')$ just share the vertex $v$ we once again  find $y$ such that  $(v+y|v_1, v'_2),(v+y|v'_1, v'_2) \in X_{sq}$ as above. 
 We then have 
 \[
  G_v(v_1, v_2, v_3)= G_v(v_1, v_2, y)=  G_v(v'_1, v'_2, y)=G_v(v'_1, v'_2, v_3).
 \]
 \end{proof}

\begin{corollary}[Definition of the extension on $V_0$]\label{def1} Let $v \in V_0$. By Lemma \ref{def-extension-X_a} the function $G_v(\bar v)$ is constant  $q^{-\Omega(r(X))}$ a.e. $\bar v$ such that $(v|\bar v)$ has all vertices but $v$ in $X_{sq}$.  Define $g(v)$ to be this value. 
\end{corollary}

Observe that $g$ is an extension of $h$, since for $w \in X_{sq}$ we have $G_v(\bar v)$ is constant for all cube $(v|\bar v)$ with all vertices in $X_{sq}$. Thus $g$ vanishes on $q^{-\Omega(r(X))}$ a.e cube in $V_0$ with at most one vertex outside $X_{sq}$. 
\begin{proposition} $g$ vanishes on $q^{-\Omega(r(X))}$a.e. $3$-cubes in $V_0$. 
\end{proposition}

\begin{proof} We prove this in steps, increasing  one by one the number of vertices in $V_0$ outside of $X_{sq}$.\\

Suppose $(v|\bar v)$ is a non degenerate cube in $V_0$ with all vertices but $v, v+v_1$ in $X_{sq}$.  Find a $y$ such that 
$(y+v|v_2, v_3)$ has all vertices in in $X_{sq}$.  By Lemma \ref{opposite1}, there are $q^{-O(1)}$ many such choices.  Now
\[
(v|\bar v) = (v|y, v_2,v_3)- (v+v_1|y-v_1, v_2,v_3)
\]
a difference of two $3$ cubes wth at most one vertex outside $X_{sq}$. Thus $g_3(v|\bar v)=0$ for $q^{-r}$ such cubes. \\

Now suppose $(v|\bar v)$ is a non degenerate cube in $V_0$, with all vertices but $v, v+v_1, v+v_2, v+v_1+v_2$ in $X_{sq}$.   Find a $y$ such that 
$(y+v|v_2, v_3)$ has all vertices in in $X_{sq}$, and 
$(v|\bar v) = (v|y, v_2,v_3)- (v+v_1|y-v_1, v_2,v_3)$
is 
a difference of two $3$ cubes wth at most two vertex outside $X_{sq}$. Thus $g_3(v|\bar v)=0$ for $q^{-r}$ such cubes. \\

Finally suppose $(v|\bar v)$ is a non degenerate cube in $V_0$.   Find a $y$ such that 
$(y+v|v_2, v_3)$ has all vertices in in $X_{sq}$, and  then 
$(v|\bar v) = (v|y, v_2,v_3)- (v+v_1|y-v_1, v_2,v_3)$
is 
a difference of two $3$ cubes wth at most a face outside $X_{sq}$. Thus $g_3(v|\bar v)=0$ for $q^{-r}$ such cubes. \\
\end{proof}

\begin{corollary}
There exists a function $\ti g:V_0 \to k$ such that   $\ti g$ vanishes on all  $3$-cube $(u|\bar u)$ in $V_0$, $\ti g$ agrees with $h$ on $X$ and 
agrees with $g$ on $q^{-\Omega(r)}$ a.e. $v \in V_0$. 
\end{corollary}

\begin{proof} This follows from classical polynomial testing result  (or from Proposition \ref{testing-X2}) along with the observation that the construction of $\ti g$ 
gives that $\ti g=g$ on $X_{sq}$ since $g$ vanishes on $3$ cubes in $X$. 
\end{proof}

We now have a funciiton $\ti g:V_0 \to k$ that is quadratic on $V_0$ and vanishes on all cubes in $X$. We wish to extend $\ti g$ to a quadratic function on $V$.  

\begin{proposition}\label{V0V} Let $W\subset V$ be a subspace of codimension $1$, $W=\{ v\in V| (v,v_0)=0\}$ and  such that there exists a quadratic function  $h:W \to k$ extending $f:X\cap W \to k$. Then $h$ can be extended to a quadratic function on $V$  that is the common extension of $h$ and of $f:X \to k$.
\end{proposition}

\begin{proof} Fix $x_0\in X-W$.
We can extend $h$ to a quadratic form on $V$ such that $h(x_0)=f(x_0)$.

Any $v \in V$ can be written uniquely as $v = c(v)x_0 + w(v)$ , with $w(v)\in W$. Define $q(v) = c(v)^2 + h(w(v))$, and replace $f$ with $f-q$. 
 Then $f$ vanishes on cubes in $X$ and $f|_{W \cap X} =0$, and $f(cx_0)=0$, $c\in k$.  \\
 
 Let 
 $$M=\{ v \in W|(v,x_0)=0\}, \quad N=\{ v \in V|(v,x_0)= 0\}.$$ 
 We claim that the restriction of $f$ to $(x_0+M) \cap X$ vanishes on squares: Let 
 $$b=(x_0+u_0;u_1, u_2)$$ 
 be a square in $(x_0+M)\cap X$. Since $f|_{X \cap W}=0$ and $f$ vanishes on cubes in $X$ it suffices to show that we can find an opposite face in $X \cap W$. The condition that $b \in  (x_0+M)\cap X$ gives $u_i \in W$,
 \[
 (u_0, x_0) = (u_1, x_0)_= (u_2, x_0)=0,
 \]
 and
 \[
Q(u_0)=Q(u_1)+(u_1, u_0)=Q(u_2)+(u_2, u_0)=Q(u_1+u_2)+(u_1+u_2, u_0)=0.
 \]
  We need to find $y$ such that 
  \[
  y,y+u_1,y+u_2,y+u_1+u_2 \in X \cap W.
  \]
  But this amounts to finding $y\in W$ with  
  \[
  Q(y)=Q(u_1)+(y,u_1)=Q(u_2)+(y,u_2)=0.
  \]
 This could be done by Lemma \ref{solutions}
  \ \\
 Consider the function $f$ on $x_0+M$. Then $f$ vanishes on squares in $(x_0+M) \cap X$ and $f(x_0)=0$ and thus $f$ can thus be extended to an affine function $\lambda _0$ on $x_0+M$. \\
 \ \\
Let $\lambda$ be  a linear function on $V$ whose restriction to 
$x_0+M$ is equal to  $\lambda _0$.
We replace $f$ with $f-\lambda \mu$ where   $ \mu \in V^{\vee}$ is such that $\mu|_W =0$ and $\mu(x_0)=1$. 
We now  have $f|_{W \cap X} = f|_{((cx_0+M) \cap X}=0, c\in k$. 

\begin{claim} $f|_{N \cap X}=0$.
\end{claim}
\begin{proof} Note that $x_0 \in N$ so that $N=x_0+N$. It is clear that $N=kx_0+M$ and therefore any $x\in N\cap X$ is of the form 
$x=cx_0+m$, $c\in k$, $m\in M$. Since $x\in N$ we have $m\in M\cap X$ and thererefore $f(x)=0$.
\end{proof}

We now have $f|_{N\cap X}=f|_{W\cap X}=0$ and $f$ vanishes on cubes in $X$. Let 
$$X_1 = W\cap X, \quad X_2 = N\cap X, \quad W'=W \cap N.$$

\begin{claim} For any $(x, x') \in X^2$ with  $x-x' \in W'$ then $f(x)=f(x')$. 
\end{claim}
\begin{proof}
Consider first the case when $y\in X \cap W'$. We want to place the edge $(x,x')$ is a cube in $X$ with all vertices in one of  $X_1, X_2$, for then  since $f$ vanishes on cubes and on $X_1, X_2$ then we will have  $f(x)=f(x')$. We need to find $u$ so that
\[
u, u+x-x', u+x, u+x' \in X
\]
and  $u, u+x-x' \in W$ and $u+x, u+x'  \in N$. Then the cube $(0;u,x', x-x')$ is the cube we seek. 
The above conditions amount to :
\[
Q(u)=(u, x)=(u, x')=0;  (u+x, x_0)=(u+x', x_0)=(u, v_0)=(u+x-x',v_0)=0.
\]
Since $x-x' \in W'$ this is the same as finding $u \in W$ such that 
\[
Q(u)=(u, x)=(u, x')=0;  (u+x, x_0)=(u, v_0)=0.
\]
This set of conditions is admissible by Lemma \ref{solutions}.
\ \\
Now for  $x,x' \in X$ such that $x-x'\in W'$ we can find $x''\in X$ such that $x-x'', x'-x'' \in X\cap W$: This amount to finding an $x''$ with
\[
Q(x'') =(x,x'') = (x',x'')=  (x'',x_0)= (x'',v_0)  = 0.
\]
which can be done by Lemma  \ref{solutions}.
\end{proof}
\ \\
 Any $v\in V$  can be written in the form $v=x_v +w_v$ , $x_v \in X$, $w_v \in W'$.
Thus we can extend $f$ to a function $f$ on $V$ by $f(v):= f(x_v)$, and by claim this $f$ is well defined. \\
\ \\\
We show that $f$ vanishes on $3$-cubes: Let $c=(w; v_1, v_2, v_3)$ be a  cube. 
It suffices to find $x_i$ such  $(w; x_1, x_2, x_3)$ is a cube in $X$ and $\omega \cdot (\bar v -\bar x) \in W'$ and for any $\omega \subset \2^3$. 
This amounts to finding $x_1,x_2, x_3$ with 
\[
Q(w+x_i)=(x_j, x_k)=0; \quad  (v_i-x_i, v_0)=(v_i-x_i, x_0)=0.
\] 
which is solvable by Lemma \ref{solutions}.
\end{proof}
This completer the proof of Theorem \ref{extension1}.
\end{proof}

{\em Proof of Theorem \ref{quadratic-extension}}. We constructed a quadratic function $g$ that agrees $q^{-\Omega(r(Q))}$ a.e with $f$ on $X$. Consider the function $f-g$. It is  weakly  quadratic and vanishes 
 $q^{-\Omega(r(Q))}$ a.e on $X$.  Fix $x \in X$. By Proposition \ref{solutions-X} there are $q^{-O(1)}|V|$ many $y,z$ such that $\langle x, y, z \rangle$ is an isotropic subspace.  In particular for any such $y,z$ we have $(0|x,y,z) \in C_3(X)$. Since $f-g$ vanishes $q^{-\Omega(r(Q))}$ a.e. , if $s$ is sufficiently large we can find $y,z$ such that $f-g$ vanishes on $y, z, x+y, x+z,y+z, x+y+z $ and $\langle x, y, z \rangle$ is an isotropic subspace. But then $(f-g)(x)=0$.

 \section{Extending weakly linear and weakly quadratic functions - complex case}

In this section we show how to modify our arguments in case of complex manifolds, to prove Theorems \ref{complex-multi}, Theorem \ref{quadratic-extension-complex}
We will freely use the notation from earlier sections of the paper.

We start with the following standard result (Bezout)

\begin{claim} Let $V$ be a $\mC$-vector space of dimension $n$,  $P_i:V\to k,1\leq i\leq L$ a family of homogeneous  polynomials,  $ X=\{ v|P_i(v)=0\}$. Then $dim (X)\geq n-L$.

\end{claim}

\begin{corollary} a) Let $Y\subset V^3$ be as in section \ref{counting}. Then $dim(Y)\geq 3n-4L$.

b) Let  $Y_x\subset V^3$ be as in section \ref{testing}. Then  $dim (Y_x)\geq 2n-3L$.

c) Let $Y\subset V$ be as in Proposition \ref{solutions-X}. Then $dim(Y)\geq n-mD, D=\sum_{i=1}^L d_i$.
\end{corollary}

Let $E\subset V^3$ be as in section \ref{counting}. 

\begin{claim} $ dim(E)\leq 3n-s$  if $r\geq r(s)$. 

\end{claim}
\begin{proof} Follows from Lemma \ref{1} in Appendix \ref{construct}
\end{proof}

\begin{corollary}  $E\cap Y\subset Y$ is a subvarity of positive codimension  if $r\geq r(4L+1)$.

\end{corollary}

 Now all the arguments used in the proof of Theorem \ref{finite-multi} work in the case of complex varieties.

 %%%%%%%%%%%%%%%%%%%%%%%%%%%%%%%%%%%%%%%%%%%%%%%%%%%%%%%%%%%%%%%%%%%

\section{Appendix 1} \label{appendix-solutions}

Let $k=\mF _q$, and let   $V$ be a $k$-vector space, $N=dim(V),\mP(V)$ the corresponding projective space. For any subspace $W\subset V$ we have a natural imbedding $\mP (W) \ho \mP (V)$.  \\

\begin{definition} A subset  $ Z\subset \mP (V)$ is {\it $D$-large} if $Z\cap \mP (W)\neq \emp$ for any subspace $W$ of $V$ of dimension $> D$.
\end{definition}

\begin{lemma}  $|Z|\geq |\mP (V)|/2q^{D+1}$ for any {\it $D$-large} subset 
$Z\subset \mP (V)$.  
\end{lemma}

\begin{proof} Let $Gr_N^D$ be the set of $D+1$-dimensional subspaces $W$ of $V$. As well known  $|Gr_N^D| =\binom{N}{D+1}_q$.
In particular 
 $$|Gr_N^D|/|Gr _{N-1}^{D-1}|\geq q^{N}-1/q^{D+1}-1 \ge q^{N}/2q^{D+1}$$

For any $l\in \mP (V)$ we denote  $Gr_l\subset Gr$ the set of $D+1$-planes $W$ containing $l$. The size of  $Gr _l$ does not depend on $l$ and is equal to   $|Gr_{N-1}^{D-1}|$. Since $Z$ is {\it $D$-large} we have $Gr =\bigcup _{l\in Z}Gr _l$. Therefore $|Z|\geq |Gr_N^D|/|Gr _{N-1}^{D-1}|.$ 
\end{proof}

Let $W$ be a $k$-vector space,
 $P_i:W\to k,1\leq i\leq s$ homogeneous  polynomials of  degrees $d_i, X=\{ w|P_i(w)=0\}, D:=\sum _id_i$.

\begin{claim}[\cite{ax}, Corollary to the main theorem in the end section $3$]  If $dim (W)> D$ then $X(k)\neq \{ 0\}$.
\end{claim}

Let   $V$ be a $k$-vector space,
 $P_i:W\to k,1\leq i\leq s$ be homogeneous  polynomials of  degrees $d_i\ge 1$ and let  $D:=\sum _id_i$.  Let  $\ti Y=\{ v|P_i(v)=0\}$. The subset $\ti Y$ of $V$ is homogeneous. We denote by $Y$ the corresponding subset of $\mP (V)$.

\begin{corollary} $Y$ is {\it $D$-large}.
\end{corollary}

\begin{corollary}\label{large} $|Y(k)|\geq |\mP (V)|/2q^{D+1}$.
\end{corollary}
\ \\
Let $P:V\to k$ be a homogeneous polynomial of degree $d \ge 1$ and 
$X:=\{ v|P(v)=0\}$. Fix $x_j \in X \setminus 0$ , $j=1, \ldots, m$ and define 
 $Y'$ as  the set of 
$y\in X$ such that  $y+tx_j\in X(\bar k)$ for all $t\in \bar k$, $j=1, \ldots, m$. 

\begin{lemma}
 $|Y'(k)|\geq |\mP (V)(k)|/2q^{D+1}, D:=md(d+1)/2$.
\end{lemma}

\begin{proof}
We expand
\[
 P(tx_j+y) =\sum t^iP_{ij}(y)
 \]
 where $P_{ij}(y)$  are homonogeneous polynomials of degree $d-i$. Let
 $Y=\{v| P_{ij}(y)=0\}, 0\leq i<d$, $j=1, \ldots, m$. 
Then  $Y\subset Y'$. By Corollary  \ref{large}  we have $|Y|\geq |V|/2q^{D+1}$.
\end{proof}

Let $P_t:V\to k$, with $t=1, \ldots, L$ be  homogeneous polynomials of degree $1 \le d_t \le d$ and 
$X:= \bigcap_t\{ v|P_t(v)=0\}$. Fix $x_j \in X \setminus 0$ , $j=1, \ldots, m$ and define 
 $Y'$ as  the set of 
$y\in X$ such that  $y+tx_j\in X(\bar k)$ for all $t\in \bar k$, $j=1, \ldots, m$. 

\begin{lemma}
 $|Y'(k)|\geq |\mP (V)(k)|/q^{O_{d, m, L}(1)}$.
\end{lemma}

\section{Appendix 2}\label{construct}
Let $X$ be  an  affine variety defined over $\mQ$. In this case  $X$ is defined by a system of polynomial equations $\{ P_i(x_1,...,x_R)\} $ in $\mcA ^R$ with rational coefficients. Let $N'$ be such that $N'a\in \mZ$ for any coefficient $a$ of any of polynomials $P_i$. Then 
for any field $k$ of characteristic $\geq N'$ 
we denote by $X_k$ the $k$-subvariety of the $R$-dimensional $k$-vector space  defined by ''the same '' equations where we 
 consider coefficients of $P_i$ as elements of $k$.

\begin{claim}  a)  $dim_k(X_k)$ is equal to the complex dimention of the variety $X(\mC)$ for all fields of characteristic $\gg 1$.

b) If $X$ is irreducible and smooth then  $X_k$ is also  irreducible and smooth if $char (k) \gg 1$.

\end{claim}
Let $k$ be a finite field. 
 For any $n\geq 1$ we denote by $k_n/k$ the extension of degree $n$.

\begin{claim} Let   $k$ be a finite field and 
$Y$ be a $k$-variety of dimension $D$.

a) There exists a constant $C>0$ such that $|Y(k_n)|\leq C|k_n|^D$ for all $n\in \mZ _+$.

b) There exists a constant $c>0$ and $n_0\geq 1$  such that $|Y(k_n)|\geq c|k_n|^D$ for all $n\in n_0\mZ _+$.

c) If $\lim _{n\to \infty} \frac {|Y(k_n)|}{k_n^D}=1$ then 
$Y$ has unique irreducible component of dimension $D$.
\end{claim}

\begin{remark} This Claim are based on Weil's estimates (\cite{weil}).
\end{remark}

Let $k$ be a field, $p:T \to S$ be a map of algebraic $k$-varieties. For any $s\in S(k)$ we denote by $T_s$ the  $k$-variety $p^{-1}(s)$. If $p:T \to S$ be a map of algebraic $\mQ$-
varieties then for field $k$, $char (k)\gg 1$ we have the corresponding morphism $p_k:T_k\to S_k$.

\begin{claim}\label{dim}
Let $p:T \to S$ be a map of algebraic $\mQ$-varieties such that  for any finite field $k$, $char (k) \gg 1, s\in S_k(k)$ there exists a constants $C,m$ such that $|T_s(k_n)|\leq C|k_n|^m$, $n\geq 1$. Then $dim (T_s)\leq m$ for any $s\in S(\mC )$.
\end{claim}

\begin{proof} As well known there exists a constructible $\mQ$-subvariety $Z\subset S$ such that for any  field $k$, $char (k)\gg1$ we have 
$$Z_k(k)=\{ s\in S_k(k)|dim T_s>m\}$$
Assume that $Z\neq \emp$. Then there exists a finite  field $k,char (k)\gg1$ such  that $Z_k(k)\neq \emp$. Fix $s\in Z_k$. By definition $dim(T_s)>m$. But on the other hand it follows from Claim 2 that $dim(T_s)\leq m$. This contradiction proves that $Z=\emp$.
\end{proof}

For any  $\bar d =(\{ d_i\}$, $1\leq i\leq L$ and $N\geq 1$ 
we denote by $Z= Z(\bar d,N)$ the variety of families  $\bar P=\{ P_i\}$,  $P_i:\mA ^N\to \mA$ of homogeneous polynomials of degrees $d_i$. For any $r\geq 1$ we denote by $Z^r\subset Z$ the open subvariety of families of rank of $\bar P$ is $\geq r$. \\

For any $(v',v'',v''',\bar P)$ we define 
$$Q(v',v'',v''',\bar P)=\{ v\in \mA ^N|
 P_i(v)= P_i(v+v')= P_i(v+v'+v'')= P_i(v''-v)=0\}.$$ 

We define 
$$E^r\subset \mA ^N\times  \mA ^N\times  \mA ^N\times Z^r$$ 
as the variety of quadruples  $(v',v'',v''',\bar P)$ such that $Q(v',v'',v''',\bar P)=\emp$ 
and denote by $p:E^r\to Z^r$ the natural projection.

\begin{lemma}\label{12} For any $s\geq 1$ there exists $r\geq 1$ such that $dim (E^r_z)\leq 3N-s$ for any $z\in Z^r(\mC)$
\end{lemma}
\begin{proof} Follows from Proposition \ref{E-large-X}, and the Claim \ref{dim}.
\end{proof}

\end{document}